\documentclass[12pt]{amsart}
 
\usepackage{fourier} 
\usepackage{dsfont} 
\usepackage{fontenc}
\usepackage[margin=1.0in]{geometry}

\usepackage{amssymb} 
\usepackage[utf8]{inputenc} 
\usepackage[english]{babel} 

\numberwithin{equation}{section}		% Numbering within section

\theoremstyle{plain}
\newtheorem{thm}{Theorem}[section]		% Numbering within section
\newtheorem{lem}[thm]{Lemma} 
\newtheorem{cor}[thm]{Corollary}

\theoremstyle{definition}

\theoremstyle{remark} 
 
\newtheorem{que}{Question}

\DeclareMathOperator{\BMOA}{BMOA} \DeclareMathOperator{\VMOA}{VMOA} \DeclareMathOperator{\BMO}{BMO}  \DeclareMathOperator{\mre}{Re}

\newcommand{\T}{\mathbb T}

\begin{document} 
\title{Volterra operators on Hardy spaces of Dirichlet series} 
\date{\today} 

\author{Ole Fredrik Brevig} \address{Department of Mathematical Sciences, Norwegian University of Science and Technology (NTNU), NO-7491 Trondheim, Norway} \email{ole.brevig@math.ntnu.no}

\author{Karl-Mikael Perfekt} \address{Department of Mathematical Sciences, Norwegian University of Science and Technology (NTNU), NO-7491 Trondheim, Norway} \email{kperfekt@utk.edu}
\curraddr{Department of Mathematics, The University of Tennessee, Knoxville, TN 37996, USA}

\author{Kristian Seip} \address{Department of Mathematical Sciences, Norwegian University of Science and Technology (NTNU), NO-7491 Trondheim, Norway} \email{seip@math.ntnu.no}

\thanks{The first and third author are supported by Grant 227768 of the Research Council of Norway.}


%\keywords{}
\begin{abstract}
For a Dirichlet series symbol $g(s) = \sum_{n \geq 1} b_n n^{-s}$,  the associated Volterra operator $\mathbf{T}_g$ acting on a Dirichlet series $f(s)=\sum_{n\ge 1} a_n n^{-s}$ is defined by the integral $f\mapsto -\int_{s}^{+\infty} f(w)g'(w)\,dw$. We show that $\mathbf{T}_g$ is a bounded operator on the Hardy space $\mathcal{H}^p$ of Dirichlet series with $0 < p < \infty$ if and only if the symbol $g$ satisfies a Carleson measure condition. When appropriately restricted to one complex variable, our condition coincides with the standard Carleson measure characterization of ${\operatorname{BMOA}}(\mathbb{D})$.  A further analogy with classical ${\operatorname{BMO}}$ is that $\exp(c|g|)$ is integrable (on the infinite polytorus) for some $c > 0$ whenever $\mathbf{T}_g$ is bounded. In particular, such $g$ belong to $\mathcal{H}^p$ for every $p < \infty$. We relate the boundedness of $\mathbf{T}_g$ to several other ${\operatorname{BMO}}$ type spaces: ${\operatorname{BMOA}}$ in half-planes, the dual of $\mathcal{H}^1$, and the space of symbols of bounded Hankel forms. Moreover, we study symbols whose coefficients enjoy a multiplicative structure and obtain coefficient estimates for $m$-homogeneous symbols as well as for general symbols. Finally, we consider the action of $\mathbf{T}_g$ on reproducing kernels for appropriate sequences of subspaces of $\mathcal{H}^2$. Our proofs employ function and operator theoretic techniques in one and several variables;  a variety of number theoretic arguments are used throughout the paper in our study of special classes of symbols $g$. \end{abstract}

\maketitle 

% INTRO
\section{Introduction} \label{sec:intro}
By a result of Pommerenke \cite{Pomm77}, the Volterra operator associated with an analytic function $g$ on the unit disc $\mathbb{D}$, defined by the formula
\begin{equation} \label{eq:Tgdisc}
	T_g f(z): = \int_0^z f(w)g'(w)\,dw, \qquad z \in \mathbb{D},
\end{equation}
is a bounded operator on the Hardy space $H^2(\mathbb{D})$ if and only if $g$ belongs to the analytic space of bounded mean oscillation $\BMOA(\mathbb{D})$. In view of the factorization $H^2 \cdot H^2 = H^1$ and C. Fefferman's famous duality theorem, according to which $\BMOA(\mathbb{D})$ is the dual of $H^1(\mathbb{D})$, it follows that $T_g$ is bounded if and only if the corresponding Hankel form $H_g$ is bounded, where
\[H_g(f,h) := \int_{\mathbb{T}}f(z) h(z)\overline{g(z)}\,dm_1(z), \qquad f, h \in H^2(\mathbb{D}).\]
In recent years, it has become known how to give a direct proof of the equivalence of the boundedness of  $T_g$ and $H_g$  \cite{AP12},  with no mention of bounded mean oscillation ($\BMO$) or Carleson measures, relying instead on the square function characterization of $H^1$ to show that $T_g f$ is in $H^1(\mathbb{D})$ whenever $f$ and $g$ are  in $H^2(\mathbb{D})$. Although the systematic study of $T_g$ was conducted much later than that of the Hankel form $H_g$ (see \cite{AC01, AS95}), one could now, based on this insight, easily imagine an exposition of the one variable Hardy space theory which considers the boundedness of Volterra operators \emph{before} $\BMOA$ and Hankel operators. One advantage would then be that the John--Nirenberg inequality, by Pommerenke's trick \cite{Pomm77}, has an elementary proof for functions $g$ such that $T_g$ is bounded.

This conception of Volterra operators, as objects of primary interest for understanding $\BMO$, underlies the present investigation of such operators on Hardy spaces of Dirichlet series $\mathcal{H}^p$ with $0 < p < \infty$.  The precise definition of these spaces will be given in the next section; suffice it to say at this point that every Dirichlet series $f(s) = \sum_{n\geq1} a_n n^{-s}$ in $\mathcal{H}^p$ defines an analytic function for $\mre s > 1/2$, and that $\mathcal{H}^p$ can be identified with the Hardy space $H^p(\mathbb{D}^\infty)$ of the countably infinite polydisc $\mathbb{D}^\infty$, through the Bohr lift. For a Dirichlet series symbol $g(s) = \sum_{n \geq 1} b_n n^{-s}$, we consider the Volterra operator $\mathbf{T}_g$ defined by
\begin{equation}
	\label{eq:vintop} \mathbf{T}_gf(s) := -\int_{s}^{+\infty} f(w)g'(w)\,dw, \qquad \mre s>1/2. 
\end{equation}
We denote the space of symbols $g$ such that $\mathbf{T}_g : \mathcal{H}^p \to \mathcal{H}^p$ is bounded by $\mathcal{X}_p$. The index $p=2$ is special, and we frequently write $\mathcal{X}$ instead of $\mathcal{X}_2$. 

A general question of interest in the theory of Hardy spaces of Dirichlet series is to reveal how the different roles and interpretations of $\BMO$  manifest themselves in this infinite-dimensional setting. The space of symbols generating bounded Hankel forms has been shown to be significantly larger than $(\mathcal{H}^1)^\ast$ \cite{OCS12}, and the space $(\mathcal{H}^1)^\ast$ itself also lacks many of the familiar features from the finite-dimensional setting. For instance, a function $f$ in $(\mathcal{H}^1)^\ast$ does not always belong to $\mathcal{H}^p$ for every $p<\infty$ \cite{MS11}. By Pommerenke's trick, however, it is almost immediate that the corresponding inclusion does hold for the space $\mathcal{X}$, i.e.,
\[
\mathcal{X} \subset \bigcap_{0 < p < \infty} \mathcal{H}^p.
\]
Furthermore, $(\mathcal{H}^1)^\ast$ is notoriously difficult to deal with, in part owing to the fact that $H^p(\mathbb{D}^\infty)$, viewed as a subspace of $L^p(\mathbb{T}^\infty)$, is not complemented when $p\neq2$. We shall find that the space $\mathcal{X}$ is significantly easier to manage.
 
One of our main results is that the spaces $\mathcal{X}_p$ can be characterized by a Carleson measure condition, in analogy with what we have in the classical one variable theory. In our context, the Carleson measure associated with the symbol $g$ will live on the product of $\mathbb{T}^\infty$ and a half-line. Again deferring precise definitions to the next section, we mention that this result takes the following form: The symbol $g$ belongs to $\mathcal{X}_p$ if and only if there exists a constant $C$ (depending on $g$ and $p$) such that
\[ \int_{\mathbb{T}^\infty}\int_0^\infty |f_\chi(\sigma)|^p |g'_\chi(\sigma)|^2\sigma\,d\sigma dm_\infty(\chi) \leq C \|f\|_{\mathcal{H}^p}^p \]
holds for every $f$ in $\mathcal{H}^p$. Here $m_\infty$ denotes Haar measure on $\T^\infty$, while $\chi$ is a character on $\T^\infty$ and $f_\chi(s):=\sum_{n\geq1} a_n \chi(n) n^{-s}$ for the Dirichlet series $f(s)=\sum_{n\geq1} a_n  n^{-s}$. This result, proved  in Section~\ref{sec:embed}, is based on an adaption to our setting of an ingenious argument from a recent paper of Pau \cite{Pau13}. Our Carleson measure condition gives us the opportunity to study non-trivial Carleson embeddings on the polydisc $\mathbb{D}^\infty$, see Sections~\ref{sec:embedsuff} and \ref{sec:embedlinear}. Our understanding is incomplete, but some of the questions asked are more tractable than the important embedding problem of $\mathcal{H}^p$ (see \cite[Sec.~3]{SS09}) while still being of a similar character. In the classical setting, the description in terms of Carleson measures shows that $T_g$ is bounded on $H^p(\mathbb{D})$ if and only if it is bounded on $H^2(\mathbb{D})$. We will see that our Carleson measure characterization implies that if $g$ is in $\mathcal{X}_p$, then $g$ is in $\mathcal{X}_{kp}$ for every positive integer $k$. As is typical in this setting, we have not been able to do better than this for a general symbol $g$, and the following interesting problem remains open:

\begin{que} \label{que:Hp}
	Is $\mathbf{T}_g$ bounded on $\mathcal{H}^2$ if and only if it is bounded on $\mathcal{H}^p$ for every $p<  \infty$?
\end{que}
We are able to give an affirmative answer to this question only in the case when $g$ is a linear symbol, i.e., when $g$ has non-zero coefficients only at the primes $p_j$ so that $g(s) = \sum_{j\geq1} a_{j} p_j^{-s}$.

Before proceeding to give a closer description of our results, we would like to mention another open problem related to Question~1.  In Section~\ref{sec:hankel}, we will observe that if $\mathbf{T}_g : \mathcal{H}^2 \to \mathcal{H}^2$ is bounded, then the corresponding multiplicative Hankel form is bounded. Furthermore, we will show that if $\mathbf{T}_g : \mathcal{H}^1 \to \mathcal{H}^1$ is bounded, then $g$ is in $(\mathcal{H}^1)^\ast$. Hence, if the answer to Question~\ref{que:Hp} is positive, then so is the answer to the following.
\begin{que} \label{que:XH1}
	Do we have $\mathcal{X}_2 \subset (\mathcal{H}^1)^\ast$?
\end{que}
The reverse inclusion is easily shown to be false. In fact, it is not even true when formulated for the finite-dimensional polydisc $\mathbb{D}^2$ (see Theorem~\ref{thm:finitecase}).

To give appropriate background and motivation for our general result about Carleson measures, we have chosen to begin by exploring in some detail  the distinguished space $\mathcal{X}_2$ and its many interesting facets. This will allow us to exhibit the ubiquitous presence of number theoretic arguments in our subject, which is a consequence of our operators $\mathbf{T}_g$ being defined in terms of integrals on the half-plane 
$\mre s>1/2$.  Roughly speaking, if trying to understand $\mathbf{T}_g$ at the level of the coefficients of $\textbf{T}_g f$, one has to investigate the interplay between the number of divisors $d(n)$ of an integer $n$ and its logarithm, $\log n$. One may also analyze symbols of number theoretic interest in terms of their function theoretic properties. In fact, our first interesting example of a bounded Volterra operator $\mathbf{T}_g : \mathcal{H}^p \to \mathcal{H}^p$, will be established by the result, shown in Section~\ref{sec:prelim}, that the primitive of the Riemann zeta function,
\[g(s) = -\int(\zeta(s+1)-1)\,ds = \sum_{n=2}^\infty \frac{1}{n\log{n}} n^{-s},\]
is of bounded mean oscillation on the line $\mre s = 0$. Such a $\BMO$ condition easily implies that $g$ is in $ \mathcal{X}_2$, and also that $g$ is in $\mathcal{X}_p$ for $0 < p < \infty$, once our Carleson measure condition is in place.

To close this introduction, we now describe briefly the contents of the six subsequent sections of this paper. We begin in Section~\ref{sec:prelim} by introducing the Hardy spaces $\mathcal{H}^p$ and start from the preliminary result that
$\mathcal{H}^\infty \subset \mathcal{X} \subset \bigcap_{0<p<\infty} \mathcal{H}^p$. In our setting, there is a considerable gap between  $\mathcal{H}^\infty$ and $\bigcap_{0<p<\infty} \mathcal{H}^p$, as for instance functions in $\mathcal{H}^\infty$  are bounded analytic functions in the half-plane $\mre s>0$, while functions in $\bigcap_{0<p<\infty} \mathcal{H}^p$  in general will be analytic in the smaller half-plane $\mre s>1/2$. In Section~\ref{sec:prelim}, the main point is to demonstrate how $\mathcal{X}$ can be thought of as a space of $\BMO$ functions in the classical sense. Using the notation $\mathbb{C}_\theta$ for the half-plane $\{s\,:\,\mre(s)>\theta\}$ and $\mathcal{D}$ for the class of functions expressible as a Dirichlet series in some half-plane $\mathbb{C}_\theta$, we prove that
\[ \BMOA(\mathbb{C}_0) \cap \mathcal{D} \subset \mathcal{X} \subset \BMOA(\mathbb{C}_{1/2}), \] and we also show that $e^{c|g|}$ is integrable for some positive constant $c$ whenever $g$ is in $\mathcal{X}$. 

Section~\ref{sec:multip} and Section~\ref{sec:homogen} investigate properties of $\mathcal{X}$ with no counterparts in the classical theory. After showing that the primitive of $\zeta(s+\alpha)-1$ is in $\mathcal{X}$ if and only $\alpha\ge 1$, we make in Section~\ref{sec:multip} a finer analysis by identifying and studying a scale of symbols associated with the limiting case $\alpha=1$. More specifically, we find that if we replace $p^{-1-s}$ in the Euler product for $\zeta(s+1)$ by $\lambda (\log p) p^{-1-s}$, then this new symbol is in $\mathcal{X}$ if and only if $\lambda\le 1$, the point being to nail down the exact edge for a symbol to be in $\mathcal{X}$ when its coefficients enjoy a multiplicative structure. The methods used to prove this result come from two number theoretic papers of respectively Hilberdink \cite{Hilberdink09} and G\'{a}l \cite{Gal49}.     

In Section~\ref{sec:homogen}, we deduce conditions on the coefficients $b_n$ of a symbol $g(s)=\sum_{n\geq1} b_n n^{-s}$ to be in $\mathcal{X}$. We begin by showing that a linear symbol is in $\mathcal{X}$ if and only if $g$ is in $\mathcal{H}^2$.  This leads naturally to a consideration of $m$-homogeneous symbols, i.e., symbols such that $b_n$ is nonzero only if $n$ has $m$ prime factors, counting multiplicity. We obtain optimal weighted $\ell^2$-conditions for every $m\ge 2$, showing in particular that the Dirichlet series of $g$ in general converges in $\mathbb{C}_{1/m}$ and in no larger half-plane. Letting $m$ tend to $\infty$, we find that there exists a positive constant $c$, not larger than $2\sqrt{2}$,  such that 
\[ \| \mathbf{T}_g \| \le C  \Bigg(|b_2|^2+\sum_{n=3}^\infty |b_n|^2 n e^{-c\sqrt{\log n \log\log n}}\Bigg)^{1/2} \]
holds for every $g$ in $\mathcal{X}$. These results are inspired by and will be compared with analogous results of Queff{\'e}lec et~al. \cite{BCQ06,MQ10} on Bohr's absolute convergence problem for homogeneous Dirichlet series.

Section~\ref{sec:embed} begins with our general result about Carleson measures and is subsequently concerned with a study of to what extent our results for $\mathcal{X}_2$ carry over to $\mathcal{X}_p$. As already mentioned, our understanding remains incomplete, but we will see that a fair amount of nontrivial conclusions can be drawn from our general condition.

In the last two sections, we return again to the Hilbert space setting. Section~\ref{sec:hankel} explores the relationship between $\mathbf{T}_g$, Hankel operators, and the dual of $\mathcal{H}^1$. In particular, this section gives background for what we have listed as Question 2 above. Finally, Section~\ref{sec:RPK} investigates the compactness of $\mathbf{T}_g$, with particular attention paid to the action of $\mathbf{T}_g$ on reproducing kernels. Here we return to the symbols considered in Section~\ref{sec:multip} which will allow us to display an example of a non-compact $\mathbf{T}_g$-operator. 

\subsection*{Notation} \label{sec:notation}
We will use the notation $f(x) \ll g(x)$ if there is some constant $C>0$ such that $|f(x)|\leq C|g(x)|$ for all (appropriate) $x$. If we have both $f(x) \ll g(x)$ and $g(x) \ll f(x)$, we will write $f(x)\asymp g(x)$. If 
\[\lim_{x\to\infty} \frac{f(x)}{g(x)} = 1, \]
we write $f(x) \sim g(x)$. The increasing sequence of prime numbers will be denoted by $\{p_j\}_{j\geq1}$, and the subscript will sometimes be dropped when there can be no confusion. 
Given a positive rational number $r$, we will denote the prime number factorization
\[r = p_1^{\kappa_1} p_2^{\kappa_2} \cdots p_d^{\kappa_d}\]
by $r = (p_j)^\kappa$. This associates uniquely to $r$ the finite multi-index $\kappa(r) = (\kappa_1,\,\kappa_2,\,\ldots\,)$. For $\chi$ in $\mathbb{T}^\infty$, we set $\chi(r): = (\chi_j)^\kappa$, when $r=(p_j)^\kappa$. If $r$ is an integer, say $n$, then the multi-index $\kappa(n)$ will have non-negative entries. We let $(m,n)$ denote the greatest common divisor of two positive integers $m$ and $n$. The number of prime factors in $n$ will be denoted $\Omega(n)$ (counting multiplicities) and $\omega(n)$ (not counting multiplicities), and $\pi(x)$ will denote the number of primes less than or equal to $x$. We will let $\log_k$ denote the $k$-fold logarithm so that $\log_2 x =\log\log x$, $\log_3 x =\log\log\log x$, and so on. To avoid cumbersome notation, we will use the convention that $\log_k x=1$ when $x\le x_k$, where $x_2=e^e$ and $x_{k+1}=e^{x_k}$ for $k\ge 2$.

% PRELIMINARIES
\section{The Hardy spaces $\mathcal{H}^p$, symbols of Volterra operators, and $\BMO$ in half-planes} \label{sec:prelim} 

\subsection{Hardy spaces of Dirichlet series} The Bohr lift of the Dirichlet series $f(s)=\sum_{n\geq1} a_n n^{-s}$ is the power series $\mathcal{B}f(z) = \sum_{n\geq1} a_n z^{\kappa(n)}$. For $0 < p < \infty$, we define $\mathcal{H}^p$ as the space of Dirichlet series $f$ such that $\mathcal{B}f$ is in $H^p(\mathbb{D}^\infty)$, and we set 
\[\|f\|_{\mathcal{H}^p} := \|\mathcal{B}f\|_{H^p(\mathbb{D}^\infty)} = \left(\int_{\mathbb{T}^\infty} |\mathcal{B}f(z)|^p\,dm_\infty(z)\right)^\frac{1}{p}.\]
Here $m_\infty$ denotes the Haar measure of the infinite polytorus $\mathbb{T}^\infty$, which is simply the product of the normalized Lebesgue measure of the torus $\mathbb{T}$ in each variable. Note that for $p=2$, we have
\[\|f\|_{\mathcal{H}^2} = \left(\sum_{n=1}^\infty |a_n|^2\right)^\frac{1}{2}.\]
We refer to \cite{QQ13} (or to \cite{Bayart02,HLS97}) for a treatment of the properties of $\mathcal{H}^p$, describing briefly the basic results we require below. For a character $\chi$ in $\mathbb{T}^\infty$, we define
\[f_\chi(s) := \sum_{n=1}^\infty a_n \chi(n) n^{-s}.\]
For $\tau$ in $\mathbb{R}$, the vertical translation of $f$ will be denoted by $f_\tau(s) := f(s+i\tau)$. It is well-known (see \cite[Sec.~2]{HLS97}) that if $f$ converges uniformly in some half-plane $\mathbb{C}_\theta$, then $f_\chi$ is a normal limit of vertical translations $\{f_{\tau_k}\}_{k\geq1}$ in $\mathbb{C}_\theta$.

The conformally invariant Hardy space $H_{\operatorname{i}}^p(\mathbb{C}_\theta)$ consists of holomorphic functions in  $\mathbb{C}_\theta$ that are finite with respect to the norm given by
\[\|f\|_{H_{\operatorname{i}}^p(\mathbb{C}_\theta)} := \sup_{\sigma>\theta} \left(\frac{1}{\pi}\int_{\mathbb{R}} |f(\sigma+it)|^p\,\frac{dt}{1+t^2}\right)^\frac{1}{p}.\]
The following connection between $\mathcal{H}^p$ and $H_{{\operatorname{i}}}^p(\mathbb{C}_0)$ can be obtained from Fubini's theorem:
\begin{equation}
	\label{eq:avgrotemb} \|f\|_{\mathcal{H}^p}^p = \int_{\mathbb{T}^\infty} \|f_\chi\|^p_{H^p_{\operatorname{i}}(\mathbb{C}_0)} \, dm_\infty(\chi). 
\end{equation}
Based on \eqref{eq:avgrotemb}, one can deduce Littlewood--Paley type expressions for the norms of $\mathcal{H}^p$. This was first done for $p=2$ in \cite[Prop.~4]{Bayart03}, and later for $0<p < \infty$ in \cite[Thm.~5.1]{BQS14}, where the formula
\begin{equation}
	\label{eq:LPp} \|f\|_{\mathcal{H}^p}^p \asymp |f(+\infty)|^p + \int_{\mathbb{T}^\infty} \int_{\mathbb{R}}\int_0^\infty |f_\chi(\sigma+it)|^{p-2}|f_\chi'(\sigma+it)|^2 \sigma d\sigma\frac{dt}{1+t^2}dm_\infty(\chi)
\end{equation}
was obtained. When $p=2$, we have equality between the two sides of \eqref{eq:LPp}. We note in passing that this fact can be used to relate $\mathcal{X}$ to $\mathcal{H}^\infty$, the space of bounded Dirichlet series in $\mathbb{C}_0$ endowed with the norm
\[\|f\|_\infty := \sup_{\sigma>0}|f(s)|, \qquad s = \sigma +it.\]
Indeed, let $M_g$ denote the operator of multiplication by $g$ on $\mathcal{H}^2$, and recall the result that $M_g$ is bounded if and only if $g$ is in $\mathcal{H}^\infty$, with $\|M_g\| = \|g\|_\infty$ (see \cite[Thm.~3.1]{HLS97}). Since $(fg)'=f'g  +(\mathbf T_g f)'$, it then follows from the Littlewood--Paley formula and the triangle inequality that 
\begin{equation} \label{eq:TgHi}
\| \mathbf{T}_g\| \le 2 \| g\|_{\infty} \end{equation}
and consequently $\mathcal{H}^\infty\subset \mathcal{X}$.

Dirichlet series in $\mathcal{H}^p$ for $0< p<\infty$ are however generally convergent only in $\mathbb{C}_{1/2}$. In this half-plane, we have the following local embedding from \cite[Thm.~4.11]{HLS97}. For every $\tau$ in $\mathbb{R}$, 
\begin{equation}
	\label{eq:localemb} \int_{\tau}^{\tau+1}|f(1/2+it)|^2\,dt \leq C\|f\|_{\mathcal{H}^2}^2. 
\end{equation}
It is sometimes more convenient to use the equivalent formulation that 
\begin{equation}
	\label{eq:localiemb} \|f\|_{H^2_{\operatorname{i}}(\mathbb{C}_{1/2})}^2\leq\widetilde{C}\|f\|_{\mathcal{H}^2}^2. 
\end{equation}
It is interesting to compare \eqref{eq:avgrotemb} and \eqref{eq:localiemb}. These formulas illustrate why both half-planes $\mathbb{C}_0$ and $\mathbb{C}_{1/2}$ appear in the theory of the Hardy spaces $\mathcal{H}^p$. It will become apparent in what follows that both half-planes show up in a natural way also in the study of Volterra operators.

\subsection{$\BMO$ spaces in half-planes} The space $\BMOA(\mathbb{C}_\theta)$ consists of holomorphic functions in the half-plane $\mathbb{C}_\theta$ that satisfy 
\[\|g\|_{\BMO(\mathbb{C}_\theta)} := \sup_{I \subset \mathbb{R}} \frac{1}{|I|}\int_{I}\left|f(\theta+it)-\frac{1}{|I|}\int_I f(\theta+i\tau)\,d\tau\right|\,dt < \infty. \]
We let as before $\mathcal{D}$ denote the space of functions that can be represented by Dirichlet series in some half-plane. The abscissa of boundedness of a given $g$ in $\mathcal{D}$,  denoted by $\sigma_b$, is the smallest real number such that $g(s)$ has a bounded analytic continuation to $\mre(s)\geq\sigma_b+\delta$ for every $\delta>0$. A classical theorem of Bohr \cite{Bohr13} states that the Dirichlet series $g(s)$ converges uniformly in $\mre(s)\geq\sigma_b+\delta$ for every $\delta>0$.
\begin{lem}
	\label{lem:chibmo} Assume that $g$ is in $\mathcal{D} \cap \BMOA(\mathbb{C}_0)$. Then 
	\begin{itemize}
		\item[(i)] $g$ has $\sigma_b\leq 0$; 
		\item[(ii)] $g_{\chi}$ is in $\BMOA(\mathbb{C}_0)$ and $\|g_\chi\|_{\BMO}=\|g\|_{\BMO}$ for every character $\chi$; 
		\item[(iii)] $g$ is in $\bigcap_{0<p<\infty} \mathcal{H}^p$ and $\exp(c|\mathcal{B}g|)$ is integrable on $\mathbb{T}^\infty$ for some $c>0$.
	\end{itemize}
\end{lem}
An interesting point is that the space $\mathcal{D} \cap \BMOA(\mathbb{C}_0)$ enjoys a stronger translation invariance, expressed by items (i) and (ii), than what the space $\BMOA(\mathbb{C}_0)$ itself does. Lemma~\ref{lem:chibmo} can also be interpreted as saying that $\mathcal{D} \cap \BMOA(\mathbb{C}_0)$ is only ``slightly larger'' than $\mathcal{H}^\infty$. 
We will later see that part (iii) of Lemma~\ref{lem:chibmo} holds whenever $\mathbf{T}_g$ is a bounded operator.
\begin{proof}
	[Proof of Lemma~\ref{lem:chibmo}] By the definition of $\sigma_b$,  there exists a positive number $M$ such that $|g(\sigma+it)|\leq M$ whenever $\sigma\geq \sigma_b+1$. Since $g$ is assumed to be in $\BMOA(\mathbb{C}_0)$, there exists a constant $C$ such that
	\[\int_{-\infty}^\infty |g(i\tau)-g(\sigma_b+1+it)| \frac{\sigma_b+1}{(\sigma_b+1)^2+(\tau-t)^2} \frac{d\tau}{\pi} \leq C. \]
	Therefore, by the triangle inequality, we find that
	\[\int_{t-\sigma_b-1}^{t+\sigma_b+1} |g(i\tau)| d\tau \leq 2(\sigma_b+1)\cdot (M+C). \]
	Writing $g$ as a Poisson integral, we see that this bound implies (i). Now (ii) follows immediately from the translation invariance of $\BMOA$, the characterization of $\BMOA$ in terms of Poisson integrals, and that $f_\chi$ is a normal limit of vertical translations of $f$ in $\mathbb{C}_0$ by (i). 	To prove (iii), we use the John--Nirenberg inequality to conclude that there is $c = c(\|g\|_{\BMO}) > 0$ and $C = C(\|g\|_{\BMO})$ such that 
	\[\big\|e^{c|g - g(1)|}\big\|_{L^1_{\operatorname{i}}(i\mathbb{R})} :=\frac{1}{\pi}\int_{\mathbb{R}} e^{c|g(it)-g(1)|}\,\frac{dt}{1+t^2} \leq C.\]
	Since $\sigma_b(g)\leq0$, we know that $g$ is absolutely convergent at $s=1$, so
	\[\big\|e^{c|g - g(1)|}\big\|_{L^1_{\operatorname{i}}(i\mathbb{R})} \asymp \big\|e^{c|g|}\big\|_{L^1_{\operatorname{i}}(i\mathbb{R})},\]
	where the implied constant depends on $g$, but only on the absolute value of its coefficients. In particular, we can conclude that
	\[\big\|e^{c|g_\chi|}\big\|_{L^1_{\operatorname{i}}(i\mathbb{R})} \leq \widetilde{C},\]
	for every $\chi\in\mathbb{T}^\infty$, and $\widetilde{C}$ does not depend on $\chi$, by (ii). Integrating over $\mathbb{T}^\infty$ and using Fubini's theorem as in \eqref{eq:avgrotemb} allows us to conclude that $\exp(c|\mathcal{B}g|)$ is in $L^2(\mathbb{T}^\infty)$, which also implies that $g$ is  in $\bigcap_{0<p<\infty} \mathcal{H}^p$.
\end{proof}

We require the following standard result, which can be extracted from \cite[Sec.~VI.1]{Garnett}.
\begin{lem}
	\label{lem:carlesonbmohalfplane} Let $g$ be holomorphic in $\mathbb{C}_\theta$. Then the measure
	\[\mu_g(s) = |g'(\sigma+it)|^2 \left(\sigma-\theta\right)\,d\sigma\,\frac{dt}{1+t^2}\]
	is Carleson for $H_{\operatorname{i}}^p(\mathbb{C}_\theta)$ if and only if $g$ is in $\BMOA(\mathbb{C}_\theta)$, and $\|\mu_g\|_{\mathrm{CM}(H^p_{\operatorname{i}})} \asymp \|g\|^2_{\BMO(\mathbb{C}_0)}$. 
\end{lem}

We are now ready for a first result, saying that for the boundedness of $\mathbf{T}_g$ it is sufficient that $g$ is in $\BMOA(\mathbb{C}_0)$ and necessary that it is in $\BMOA(\mathbb{C}_{1/2})$. On the one hand, it is a preliminary result, following rather directly from the available theory of $\mathcal{H}^2$, outlined above. On the other hand, as we  shall see in Section~\ref{sec:multip} and Section~\ref{sec:homogen}, $\mathbb{C}_0$ and $\mathbb{C}_{1/2}$ are the extremal half-planes of convergence for symbols $g$ inducing bounded Volterra operators. 
\begin{thm}
	\label{thm:bmocond} Let\/ $\mathbf{T}_g$ be the operator defined in \eqref{eq:vintop} for some Dirichlet series $g$ in $\mathcal{D}$. 
	\begin{itemize}
		\item[(a)] If $g$ is in $\BMOA(\mathbb{C}_0)$, then $\mathbf{T}_g$ is bounded on $\mathcal{H}^2$. 
	\end{itemize}
	Suppose that $\mathbf{T}_g$ is bounded on $\mathcal{H}^2$. Then, 
	\begin{itemize}
		\item[(b)] $g$ satisfies condition {\normalfont (iii)} from Lemma~\ref{lem:chibmo}; 
		\item[(c)] $g$ is in $\BMOA(\mathbb{C}_{1/2})$. 
	\end{itemize}
\end{thm}
\begin{proof}
	We apply \eqref{eq:LPp} to $\mathbf{T}_gf$ and use Lemmas~\ref{lem:chibmo} and \ref{lem:carlesonbmohalfplane}. Since $(fg')_\chi = f_\chi g_\chi'$ we find that 
	\begin{align*}
		\|\mathbf{T}_g\|_{\mathcal{H}^2}^2 &\asymp \int_{\mathbb{T}^\infty} \int_\mathbb{R} \int_0^\infty |(fg')_\chi(\sigma+it)|^2 \sigma\,d\sigma\,\frac{dt}{1+t^2}\,dm_\infty(\chi) \\
		& \ll \int_{\mathbb{T}^\infty} \|f_\chi\|_{H_\mathrm{i}^2(\mathbb{C}_0)}^2 \|g_\chi\|_{\BMO(\mathbb{C}_0)}^2\,dm_\infty(\chi) = \|f\|_{\mathcal{H}^2}^2 \|g\|_{\BMO(\mathbb{C}_0)}^2. 
	\end{align*}
	This completes the proof of (a).
	
	For (b), we first observe that $\mathbf{T}_g1=g$, so that $g$ is in $\mathcal{H}^2$. Applying $\mathbf{T}_g$ inductively to the powers $g^n$, for $n =1,2,\ldots$, we get that
	\[\| g^n \|_{\mathcal{H}^2} \leq \| \mathbf{T}_g\|^n n! . \]
	Using this and the triangle inequality, we obtain
	\[ \big\| e^{c|\mathcal{B} g|}\big\|_{L^1(\T^\infty)}^{1/2}=\big\| e^{c|\mathcal{B} g|/2} \big\|_{L^2(\T^\infty)}\le 
	\sum_{n=0}^\infty \Big(\frac{ c  \| \mathbf{T}_g\|}{2}\Big)^n  , \]
	which implies that  $e^{c|\mathcal{B} g|}$ is integrable whenever $c<2/ \| T_g\|$.
	
	To prove (c), we use the Littlewood--Paley formula for $H^2_{\operatorname{i}}(\mathbb{C}_{1/2})$ and \eqref{eq:localiemb} to see that
	\[\int_\mathbb{R} \int_{1/2}^\infty |f(\sigma+it)|^2\, |g'(\sigma+it)|^2 \left(\sigma-\frac{1}{2}\right)\,d\sigma\,\frac{dt}{1+t^2} \asymp \|\mathbf{T}_g f\|_{H^2_{\operatorname{i}}(\mathbb{C}_{1/2})}^2 \ll \|\mathbf{T}_gf\|_{\mathcal{H}^2}^2 \leq \|\mathbf{T}_g\|^2 \|f\|_{\mathcal{H}^2}^2.\]
	This means that
	\[\mu_g(s) = |g'(\sigma+it)|^2 \left(\sigma-\frac{1}{2}\right)\,d\sigma\,\frac{dt}{1+t^2}\]
	is a Carleson measure for $\mathcal{H}^2$ in $\mathbb{C}_{1/2}$. By \cite[Thm.~3]{OS12}, this implies that $\mu_g(s)$ is a Carleson measure for the non-conformal Hardy space $H^2(\mathbb{C}_{1/2})$, which as in Lemma~\ref{lem:carlesonbmohalfplane} means that $h(s): = g(s)/(s+1/2)$ is in $\BMO(\mathbb{C}_{1/2})$. Indeed, we have proved that $\|h\|_{\BMO(\mathbb{C}_{1/2})} \ll \|\mathbf{T}_g\|$.
	
	Let us show that the factor $(s+1/2)^{-1}$ can be removed, so that $g$ is  in fact in $\BMOA(\mathbb{C}_{1/2})$. We note first that if $|I|\geq1$, then it follows from the local embedding \eqref{eq:localemb} that
	\[\int_I |g(1/2+it)|^2\,dt \ll |I|\cdot\|g\|_{\mathcal{H}^2}^2,\]
	since $g$ is in $\mathcal{H}^2$ by (b). Hence we only need to consider intervals of length $|I|<1$. For a character $\chi$ in $\mathbb{T}^\infty$, we define
	\[h_\chi(s) := \frac{g_\chi(s)}{s+1/2}.\]
	Clearly, $\|\mathbf{T}_{g_\chi}\|=\|\mathbf{T}_g\|$ for every $\chi$ in $\mathbb{T}^\infty$. This means that
	\[\sup_{\chi\in\mathbb{T}^\infty} \|h_\chi\|_{\BMO(\mathbb{C}_{1/2})} \ll \|\mathbf{T}_g\|.\]
	In particular, the $\BMO$-norm of $h$ is uniformly bounded under vertical translations of $g$, so that we only need to consider intervals $I=[0,\tau]$ for $\tau<1$. On this interval, $(s+1/2)^{-1}$ and its derivative is bounded from below and above. It follows that $g$ is in $\BMO(\mathbb{C}_{1/2})$.  
\end{proof}

Combined with a result from \cite{HLS97}, part (b) of Theorem~\ref{thm:bmocond} yields the following result.

\begin{cor} \label{cor:chiest} If  $\mathbf{T}_g$ is bounded on $\mathcal{H}^2$, then for almost every character $\chi$ on $\mathbb{T}^\infty$, there is a constant $C$ such that 
	\begin{equation} \label{eq:subtle} |g_{\chi}(\sigma+i t)| \le C \log \frac{1+|t|}{\sigma}  \end{equation}
	holds in the strip $0<\sigma\le 1/2$. 
\end{cor}
\begin{proof}
We assume that $\mathbf{T}_g$ is bounded on $\mathcal{H}^2$. Then by part (b) of Theorem~\ref{thm:bmocond},  there exists a positive number $c$ such that the four functions $e^{\pm cg}$ and $e^{\pm i cg}$ are in $\mathcal{H}^2$. Now let $f$ be any of these four functions. Then \cite[Thm.~4.2]{HLS97} shows that, for almost every character $\chi$,  there exists a constant $C$ (depending on $\chi$) such that
\[ |f_\chi(\sigma+i t)-f(+\infty)| \le C \frac{1+\sqrt{|t|}}{\sigma}\]
for every point $\sigma+it$ in $\mathbb{C}_0$.  Combining the acquired estimates for the four functions  $e^{\pm cg}$ and $e^{\pm i cg}$ and taking logarithms, we obtain the desired result.
	\end{proof}
Our bound \eqref{eq:subtle} shows that almost surely $|g_\chi|$ grows at most as general functions in  $\BMOA(\mathbb{C}_0)$ at the boundary of $\mathbb{C}_0$. It would be interesting to know if this result could be strengthened. For instance, is it true that $g_{\chi}$ almost surely satisfies the $\BMO$ condition locally, say on finite intervals, whenever $\mathbf{T}_g$ is bounded on $\mathcal{H}^2$? Note that we cannot hope to have the stronger result that $g_{\chi}$ is almost surely in  $\BMOA(\mathbb{C}_0)$. Indeed, the proof of part (a) of Theorem~\ref{thm:bmocond} gives that if $g_{\chi}$ is in 
	$\BMOA(\mathbb{C}_{\theta})$ for one character $\chi$, then this holds for all characters $\chi$. In view of this fact and what will be shown in Section~\ref{sec:homogen},  $g_\chi$ will in general be in $\BMOA(\mathbb{C}_{1/2})$ and in none of the  other spaces	$\BMOA(\mathbb{C}_{\theta})$ for $0\le \theta < 1/2$.

\subsection{An unbounded Dirichlet series in $\BMO$} The canonical example of an unbounded function in $\BMO(\mathbb{R})$ is $\log|t|$, the primitive of $1/t$. The Riemann zeta function $\zeta(s)$ is a meromorphic function with one simple pole, at $s=1$. We now show that the primitive of $-(\zeta(s) - 1)$ has bounded mean oscillation on the line $\sigma = 1$. In view of Theorem~\ref{thm:bmocond}, this supplies us with an example of a bounded $\mathbf{T}_g$-operator.
\begin{thm}
	\label{thm:zetabmo} The Dirichlet series
	\[g(s) := \sum_{n=2}^\infty \frac{1}{n\log{n}}n^{-s}.\]
	is in $\BMOA(\mathbb{C}_0)$. 
\end{thm}
\begin{proof}
	We will show that $g$ is in $\BMOA(\mathbb{C}_{\varepsilon})$, with $\BMO$-norm uniformly bounded in $\varepsilon>0$. Since $g(s-1/2)$ is in $\mathcal{H}^2$, we can use the local embedding as in the proof of Theorem~\ref{thm:bmocond} (c) to conclude that $g$ satisfies the $\BMO$-condition for intervals of length $|I|\geq1$.
	
	Focusing our attention on short intervals, we fix a real number $a$ and $0 < T < 1$ and set
	\[c := \sum_{\log n < 1/T} \frac{1}{n^{1+\varepsilon}\log{n}} n^{-ia}. \]
	To prove the theorem, we will show that
	\[ \int_a^{a+T} |g(\varepsilon + it) - c |^2 \, dt \leq CT \]
	where $C$ is a universal constant.
	
	Notice first that
	\[ \int_a^{a+T} |g(\varepsilon + it) - c |^2 \, dt = \int_0^T |\widetilde{g}(\varepsilon + it) - c |^2 \, dt, \]
	where
	\[\widetilde{g}(s) := \sum_{n=2}^\infty \frac{n^{-ia}}{n\log{n}}n^{-s}.\]
	Accordingly, set $b_n := n^{-ia}/(n\log{n})$. Then we have that 
	\begin{multline*}
		\left(\int_a^{a+T} |g(\varepsilon + it) - c |^2 \, dt\right)^{1/2} \leq \\
		\left(\int_0^T \Bigg|\sum_{\log n < 1/T} b_n n^{-\varepsilon}(n^{-it} - 1) \Bigg|^2 \, dt\right)^{1/2} + \left(\int_0^T \Bigg|\sum_{\log n > 1/T} b_n n^{-\varepsilon}n^{-it} \Bigg|^2 \, dt\right)^{1/2} .
	\end{multline*}
	To deal with the second term, we use the local embedding \eqref{eq:localemb} in a similar manner as above, using now that
	\[\int_0^T |f(1/2 + \varepsilon + it)|^2 \, dt \ll \|f\|_{\mathcal{H}^2}^2\]
	in this case, since $T < 1$. This gives us that 
	\begin{equation*}
		\int_0^T \Bigg|\sum_{\log n > 1/T} b_n n^{-\varepsilon}n^{-it} \Bigg|^2 \, dt \leq \sum_{\log n > 1/T} n |b_n|^2 \ll T, 
	\end{equation*}
	as desired.
	
	For the first term, we compute: 
	\begin{equation}
		\label{eq:firstpartineq} \int_0^T \Bigg|\sum_{\log n < 1/T} b_n n^{-\varepsilon}(n^{-it} - 1) \Bigg|^2 \, dt = \sum_{\substack{\log m < 1/T \\
		\log n < 1/T}} b_n \overline{b_m} (mn)^{-\varepsilon} h_{mn}(T), 
	\end{equation}
	where
	\[ h_{mn}(T) := \frac{(n/m)^{-iT} - 1}{i \log \frac{m}{n}} - \frac{n^{-iT} - 1}{i \log \frac{1}{n}} - \frac{(1/m)^{-iT} - 1}{i \log m} + T. \]
	We write $h_{mn}$ as a Taylor series in $T$, whence
	\[ h_{mn}(T) = \sum_{k=3}^\infty d_{mn}^k T^k, \]
	where
	\[ d_{mn}^k := \frac{(-i)^{k-1}}{k!} \left( \left(\log \frac{n}{m}\right)^{k-1} - (\log n)^{k-1} - \left(\log \frac{1}{m}\right)^{k-1} \right). \]
	The point is that in the coefficient $d_{mn}^k$, the terms of order $(\log m)^{k-1}$ and $(\log n)^{k-1}$ cancel. Estimating the remaining terms in a crude manner, we have that
	\[ |d_{mn}^k| \ll \frac{2^k}{k!} \sum_{j=1}^{k-2} (\log m)^{j} (\log n)^{k-j-1}. \]
	Note that for $1 \leq j \leq k-2$, we have
	\[ T^k \sum_{\substack{\log m < 1/T \\
	\log n < 1/T}} |b_n||b_m| (\log m)^j (\log n)^{k-j-1} \ll T. \]
	We observe that this inequality fails if $j=0$ or $j = k-1$, corresponding to the terms which disappear from $d^k_{mn}$.
	
	Combining these estimates with \eqref{eq:firstpartineq} we obtain 
	\[\int_0^T \Bigg|\sum_{\log n < 1/T} b_n n^{-\varepsilon}(n^{-it} - 1) \Bigg|^2 \, dt \ll T\]
	also for the first term, concluding the proof. 
\end{proof}

% MULTIPLICATIVE
\section{Multiplicative symbols} \label{sec:multip} In this section, we study symbols of the form 
\begin{equation}
	\label{eq:multsymb} g(s) =  \sum_{n=2}^\infty \frac{\psi(n)}{\log{n}} n^{-s}, 
\end{equation}
where $\psi(n)$ is a positive multiplicative function. We know from the previous section that if $\psi(n) = n^{-1}$, then $g$ is in $\BMOA(\mathbb{C}_0)$, and therefore $g$ is in $\mathcal{X}$. We begin by considering the distinguished case when the function $\psi(n)$ corresponds to horizontal shifts of the Riemann zeta function. To be more precise, our first task will be to show that $g$ is not in $\mathcal{X}$ when $g$ is the function in $\BMOA(\mathbb{C}_{1-\alpha})$ with coefficients given by $\psi(n) = n^{-\alpha}$ and $1/2 \leq \alpha < 1$. In particular, this means that the Dirichlet series $g(s) = \sum_{n\geq2}1/(\sqrt{n}\log{n}) n^{-s}$, identified in \cite{BPSSV} as the symbol of the multiplicative analogue of Hilbert's matrix and shown there to generate a bounded multiplicative Hankel form, is indeed far from belonging to $\mathcal{X}$, as it corresponds to the case $\alpha=1/2$.

In this section and the next, we will be working at the level of coefficients. Observe that if $f(s) = \sum_{n\geq 1} a_n n^{-s}$ and $g(s) = \sum_{n \geq 2} b_n/(\log n) n^{-s}$, then
\[\mathbf{T}_gf(s)= \sum_{n=2}^\infty \frac{1}{ \log n} \Bigg(\sum_{k|n \atop k<n} a_k b_{n/k} \Bigg) n^{-s}. \]
Since the operator \[a_1+\sum_{n=2}^\infty a_n n^{-s}\quad \mapsto \quad a_1+\sum_{n=2}^\infty \frac{a_n}{\log{n}}n^{-s}\] is trivially bounded and even compact on $\mathcal{H}^2$, we will sometimes tacitly replace $\mathbf{T}_g$ with $\widetilde{\mathbf{T}}_g$,
\[\widetilde{\mathbf{T}}_gf(s) := \sum_{n=2}^\infty \frac{1}{ \log n} \Bigg(\sum_{k|n} a_k b_{n/k} \Bigg) n^{-s}, \]
where it is understood that $b_1 = 1$. 
\begin{thm}
	\label{thm:zeta} $\mathbf{T}_g$ is unbounded when $g$ is the primitive of $\zeta(s+\alpha)-1$ and $\alpha<1$. 
\end{thm}
\begin{proof}
	If $f(s)=\sum_{n\geq1} a_n n^{-s}$, then with the convention just described, we have that
	\[\mathbf{T}_g f(s)= \sum_{n=2}^\infty \frac{1}{n^{\alpha} \log n} \sum_{k| n} a_k k^\alpha n^{-s}.\]
	We now choose $f(s)=\prod_{j=1}^{J} (1+p_j^{-s})$, which satisfies $\| f\|_{\mathcal{H}^2}=2^{J/2}$. Let $\mathcal{J}$ be a subset of $\{1,...,J\}$. 
	
	Choosing $n=n_{\mathcal{J}}$, where
	\[n_{\mathcal{J}}:=\prod_{j\in \mathcal{J}} p_j, \]
	we see that
	\[\sum_{k|n_{\mathcal{J}}} a_k k^\alpha = n_{\mathcal{J}}^{\alpha} \prod_{j\in \mathcal{J}} (1+p_j^{-\alpha}).\]
	It follows that
	\[\big\|\mathbf{T}_g f\big\|_{\mathcal{H}^2}^2=\sum_{\mathcal{J}\neq \emptyset} \frac{1}{(\log n_{\mathcal{J}})^2} \prod_{j\in \mathcal{J}} (1+p_j^{-\alpha})^2,\]
	which gives
	\[\big\|\mathbf{T}_g f \big\|_{\mathcal{H}^2}^2\geq 2^{J-1} \min_{|\mathcal{J}|\geq J/2} \frac{1}{(\log n_{\mathcal{J}})^2} \prod_{j\in \mathcal{J}} (1+p_j^{-\alpha})^2.\]
	We conclude that
	\[\big\|\mathbf{T}_g f \big\|_{\mathcal{H}^2}^2\gg e^{c J^{1-\alpha}(\log J)^{-\alpha}} \| f\|_{\mathcal{H}^2}^2 \]
	for an absolute constant $c$. 
\end{proof}
The preceding clarification of the case of horizontal shifts of primitives of the Riemann zeta function motivates a more careful examination of what we need to require from the multiplicative function $\psi(n)$ in \eqref{eq:multsymb} for $g$ to belong to $\mathcal{X}$. We will now see that a surprisingly precise answer can be given if we make a slight modification of the Euler product associated with $\zeta(s)$. 

We will need the following simple decomposition of bounded $\mathbf{T}_g$-operators. Let $M_{h,x}$ denote the truncated multiplier associated with $h(s)=\sum_{n\geq1} c_n n^{-s}$ and $x\geq 1$:
\[M_{h,x}f(s):=\sum_{n\leq x} \Bigg(\sum_{k|n} c_k a_{n/k}\Bigg) n^{-s}, \]
where $f(s)=\sum_{n\geq1} a_n n^{-s}$. We observe that $M_{h,x}$ acts boundedly on $\mathcal{H}^2$ for every Dirichlet series $h$, but the point of interest is to understand how the norm of $M_{h,x}$ grows with $x$. Truncated multipliers are linked to $\mathbf{T}_g$ by the following lemma. 
\begin{lem}
	\label{dyadic} Suppose that $\mathbf{T}_g$ acts boundedly on $\mathcal{H}^2$. Then
	\[ \frac{3}{4} \sum_{k=0}^\infty 4^{-k} \big\| M_{g',{e}^{2^k}} f\big\|_{\mathcal{H}^2}^2 \leq \| \mathbf{T}_g f \|_{\mathcal{H}^2}^2 \leq 4 \sum_{k=0}^\infty 4^{-k} \big\| M_{g',{e}^{2^k}} f\big\|_{\mathcal{H}^2}^2 \]
	for every $f$ in $\mathcal{H}^2$. 
\end{lem}
\begin{proof}
	We start from the expression
	\[ \|\mathbf{T}_g f \|_{\mathcal{H}^2}^2 =\sum_{n=2}^\infty \frac{1}{(\log n)^2} \Bigg|\sum_{k|n} b_k (\log k) a_{n/k}\Bigg|^2, \]
	which we split into blocks in the following way:
	\[ \sum_{k=0}^{\infty} \frac{1}{4^k} \sum_{{e}^{2^{k-1}}< n \leq {e}^{2^k}} \Bigg|\sum_{k|n} b_k (\log k) a_{n/k}\Bigg|^2 \leq \|\mathbf{T}_g f \|_{\mathcal{H}^2}^2 \leq 4 \sum_{k=0}^{\infty} \frac{1}{4^k} \sum_{{e}^{2^{k-1}}< n \leq {e}^{2^k}} \Bigg|\sum_{k|n}b_k (\log k) a_{n/k}\Bigg|^2. \]
	The upper bound is immediate from the right inequality, and the lower bound follows from the left inequality and the fact that
	\[\sum_{e^{2^{k-1}}< n \leq {e}^{2^k}} \Bigg|\sum_{k|n} b_k a_{n/k}\Bigg|^2=\big\|M_{g',{e}^{2^k}} f \big\|_{\mathcal{H}^2}^2 - \big\|M_{g',{e}^{2^{k-1}}} f \big\|_{\mathcal{H}^2}^2.\qedhere\]
\end{proof}
The preceding lemma, which says that $\mathbf{T}_g$ is bounded whenever the norm of $M_{g',x}$ grows roughly as $\log x$, connects the study of $\mathbf{T}_g$ to the truncated multipliers considered by Hilberdink \cite{Hilberdink09} in a purely number theoretic context. Based on this observation, we shall now present a natural scale of multiplicative symbols $g_\lambda$, where $0 < \lambda < \infty$, such that $g_\lambda$ induces a bounded $\mathbf{T}_g$-operator if and only if $\lambda \leq 1$. We shall later see, in Section~\ref{sec:RPK}, that $\mathbf{T}_{g_\lambda}$ is non-compact for the pivotal point $\lambda=1$.
\begin{thm}
	\label{thmsuf} For $0 < \lambda < \infty$, let $g$ be the Dirichlet series \eqref{eq:multsymb}, where $\psi(n)$ is the completely multiplicative function defined on the primes by $\psi(p): = \lambda p^{-1}(\log p)$. Then $\mathbf{T}_g$ is bounded if and only if $\lambda\leq1$. 
\end{thm}
\begin{proof}
	We begin with the case $\lambda<1$, for which we adapt the proof of \cite[Thm.~2.3]{Hilberdink09}. Hence we let $\varphi(n)$ be an arbitrary positive arithmetic function and note that the Cauchy--Schwarz inequality implies that
	\[\| M_{g',x} f \|_{\mathcal{H}^2}^2 =\sum_{n\leq x} \Bigg| \sum_{d|n} \psi(d) a_{n/d} \Bigg|^2 \leq \sum_{n\leq x} \sum_{d|n} \frac{\psi(d)}{\varphi(d)} \sum_{k|n} \psi(k) \varphi(k) |a_{n/k}|^2. \]
	We therefore find that 
	\begin{equation}
		\label{Hilber} \| M_{g',x} \|^2_{\mathcal{H}^2} \leq \sum_{n\leq x} \varphi(n)\psi(n) \max_{m\leq x} \sum_{d|m} \frac{\psi(m)}{\varphi(m)}. 
	\end{equation}
	We now require that $\varphi$ be a multiplicative function satisfying
	\[ \varphi(p^k):= 
	\begin{cases}
		1, & p \leq M, \\
		K \sum_{r=1}^\infty \psi(p^r), & p >M, 
	\end{cases}
	\]
	where the positive parameters $K$ and $M$ will be determined later. We find that 
	\begin{align*}
		\sum_{n\leq x} \varphi(n)\psi(n) & \leq \prod_p \Bigg(1+\sum_{k=1}^{\infty}\varphi(p^k)\psi(p^k)\Bigg) \leq \exp\Bigg(\sum_{p\leq M} \sum_{k=1}^\infty \psi(p^k) + K \sum_{p>M} \Bigg(\sum_{k=1}^\infty \psi(p^k) \Bigg)^2\Bigg) \\
		& = \exp\Bigg(\sum_{p\leq M} \frac{\lambda p^{-1}\log p}{1-\lambda p^{-1} \log p} + K \sum_{p>M} \frac{\lambda^2 p^{-2}(\log p)^2}{(1-\lambda p^{-1} \log p)^2} \Bigg). 
	\end{align*}
	By Abel summation and the prime number theorem in the form
	\[ \pi(y)= \frac{y}{\log y}+\frac{y}{(\log y)^2}+O\left(\frac{y}{(\log y)^3}\right), \]
	we infer that 
	\begin{equation}
		\label{factor1} \sum_{n\leq N} \varphi(n)\psi(n) \leq \exp\left(\lambda \log M +O(1)+O\left(K \frac{\log M}{M}\right) \right). 
	\end{equation}
	We now turn to the second factor on the right-hand side of \eqref{Hilber}. We then use that also
	\[ \Phi(m):= \sum_{d|m} \frac{\psi(d)}{\varphi(d)} \]
	is a multiplicative function. We observe that
	\[ \Phi(p^k)=\sum_{r=0}^k \frac{\psi(p^r)}{\varphi(p^r)} \leq 
	\begin{cases}
		1+ \sum_{r=1}^{\infty} \psi(p^r), & p\leq M, \\
		1+K^{-1}, & p>M. 
	\end{cases}
	\]
	Consequently 
	\begin{equation}
		\label{factor2} \Phi(m)\leq \prod_{p\leq M} \Bigg(1+ \sum_{r=1}^{\infty} \psi(p^r)\Bigg) \left(1+K^{-1}\right)^{\omega(m)}\leq \exp\left(\lambda \log M + O(1) + O\left(K^{-1}\frac{\log x}{\log_2 x}\right) \right), 
	\end{equation}
	where we used that $\omega(m) \ll \log(m)/\log_2(m)$. If we now choose $M=\log x$, $K=(\log x)/\log_2 x $, and insert \eqref{factor1} and \eqref{factor2} into \eqref{Hilber}, then we find that
	\[ \| M_{g',x} \|^2 \leq C (\log x)^{2\lambda}. \]
	Finally, we invoke Lemma~\ref{dyadic} and conclude that $\mathbf{T}_g$ is bounded whenever $\lambda<1$.
	
	To show that $\mathbf{T}_g$ is bounded when $\lambda=1$ we modify the proof. In addition to the function $\varphi(n)$, we use another auxiliary function $h_x(n)$ and use the Cauchy--Schwarz inequality to obtain
	\[ \| M_{g',x} f \|_{\mathcal{H}^2}^2 =\sum_{n\leq x} \Bigg| \sum_{d|n} \psi(d) a_{n/d} \Bigg|^2 \leq \sum_{n\leq x} \sum_{d|n} \frac{\psi(d)}{\varphi(d)h_x(n/d)} \sum_{k|n} \psi(k) \varphi(k) |a_{n/k}|^2 h_x(n/k). \]
	We require from $h_x(n)$ that
	\[ \sup_m \sum_{e^{2^k}\geq m} h_{e^{2^k}}(m)<\infty. \]
	This will ensure boundedness if we can prove that
	\[ \Phi_h(m):= \sum_{d|m} \frac{\psi(d)}{\varphi(d)h_x(m/d)} \]
	enjoys the same uniform bound as that we found for $\Phi(m)$ for a suitable $h_x(n)$. To this end, we choose
	\[ h_x(n)= 
	\begin{cases}
		1, & \sqrt{x}< n \le x, \\
		\exp\left(-2\log_2 \frac{\log x}{\log n+1} \right), & 1\leq n \le \sqrt{x}, 
	\end{cases}
	\]
	which implies that
	\[ \Phi_h(m) \le \Phi(m) e^{2\log_3 x}\le \exp\left(\log_2 m + 2\log_3 x +O(1)\right). \]
	This means that in what follows, we may assume that $\log(m)\geq (\log{x})/(\log_2{x})^2$. Using again the definition of $h_x(n)$, we also obtain, for $\delta > 0$, 
	\begin{equation}
		\label{smalld} \sum_{d|m \atop m/d \ge x^{\delta}} \frac{\psi(d)}{\varphi(d)h_x(m/d)} \le \Phi(m) e^{2\log_2\frac{1}{\delta}}. 
	\end{equation}
	On the other hand, if $m=x^\beta$ with $0<\beta<1$, then arguing as before and choosing the same $M$ and $K$, we get
	\[ \Phi(m)\leq \exp\left( \log_2 x - \log \frac{1}{\beta} + O(1) \right). \]
	Hence, with $\beta = \log m/\log x$ and $\delta = \beta/2$, we find in view of \eqref{smalld} that 
	\begin{equation*}
		\sum_{d|m \atop d \leq \sqrt{m}} \frac{\psi(d)}{\varphi(d)h_x(m/d)} \le C \log x. 
	\end{equation*}
	It remains to estimate 
	\begin{equation}
		\label{remains} \sum_{d|m \atop d \ge \sqrt{m}} \frac{\psi(d)}{\varphi(d)h_x(m/d)} \le e^{2\log_3 x} \sum_{d|m \atop d \ge \sqrt{m}} \frac{\psi(d)}{\varphi(d)}. 
	\end{equation}
	Note first that
	\[ \sum_{d|m \atop d \ge \sqrt{m}} \frac{\psi(d)}{\varphi(d)} \le m^{-\varepsilon/2} \sum_{d|m} \frac{d^{\varepsilon}\psi(d)}{\varphi(d)}=: m^{-\varepsilon/2} E(m). \]
	The definition of $E(m)$ shows that, in particular, 
	\[ E(p^k)=\sum_{r=0}^k \frac{p^{\varepsilon r}\psi(p^r)}{\varphi(p^r)} \leq 
	\begin{cases}
		(1-p^{\varepsilon} \psi(p))^{-1}, & p\leq M, \\
		1+K^{-1}p^\varepsilon(1-\psi(p))/(1-p^{\varepsilon}\psi(p)) , & p>M. 
	\end{cases}
	\]
	We may assume that $\varepsilon$ is so small that the factor $(1-\psi(p))/(1-p^{\varepsilon}\psi(p))$ does not exceed $2$. Letting $\mathcal{P}$ denote an arbitrary finite set of primes $p$, we then get that 
	\begin{align*}
		\label{factor2} E(m) & \leq \prod_{p\leq \log m} \left(1- p^{\varepsilon} \psi(p)\right)^{-1} \max_{\mathcal{P}\,:\, \sum_{p\in \mathcal{P}} \log p \le \log m} \prod_{p\in \mathcal{P}} \left(1+2 K^{-1}p^{\varepsilon} \right) \\
		& \le \exp\Bigg((\log m)^{\varepsilon} \log_2 m+2K^{-1} \max_{\frac{\log x}{\log_2 x} \le p\le x} \frac{p^{\varepsilon}}{\log p}\log m +O(1)\Bigg). 
	\end{align*}
	We now choose
	\[\varepsilon:=\frac{4 \log_3 x}{\log m}.\]
	Then the latter estimate becomes 
	\[E(m) \le \exp\left((\log m)^{\varepsilon} \log_2 m+K^{-1} \frac{(\log x)^{\varepsilon}}{\log_2 x} \log m \right) \le \exp\left( \log_2 m+ O(1) \right) \le \exp\left( \log_2 x+ O(1) \right).\]
	We finally observe that the factor $m^{-\varepsilon/2}$ will take care of the term $\log_3 x$ in the exponent on the right-hand side of \eqref{remains}. 
	
	Following an insight of G\'{a}l \cite{Gal49}, we argue in the following way in order to show that $\mathbf{T}_g$ is unbounded when $\lambda>1$. We start from the fact that
	\[ \prod_{p\leq y} p = e^{y(1+o(1))} ,\]
	which is a consequence of the prime number theorem. We let $\varphi(n)$ be the multiplicative function defined by setting
	\[\varphi(p^{r}):= 
	\begin{cases}
		1, & p\leq \frac{\log x}{\log_2 x} \quad \text{and} \quad r\leq \frac{1}{2} \log_2 x, \\
		0, & \text{otherwise}. 
	\end{cases}
	\]
	Then $\varphi(n)=0$ for $n>x$ if $x$ is large enough. We set $a_n:=\varphi(n)/(\sum_n \varphi(n))^{1/2}$ and use the Cauchy--Schwarz inequality to see that
	\[ \Bigg(\sum_{n\leq x} \Bigg| \sum_{d| n} a_d \psi(n/d) \Bigg|^2\Bigg)^{1/2} \geq \frac{\sum_n \varphi(n) \sum_{d|n} \varphi(d)\psi(n/d)}{\sum_n \varphi(n)}. \]
	To simplify the writing, we set $y:=\log x/\log_2 x$ and $\ell:=\lfloor\frac12 \log_2 x\rfloor$. Then we infer from the preceding estimate that 
	\begin{align*}
		\Bigg(\sum_{n\leq N} \Bigg| \sum_{d| n} a_d \psi(n/d) \Bigg|^2\Bigg)^{1/2} & \geq \prod_{p\leq y} \frac{1+\ell + \ell \psi(p) +(\ell -1) \psi(p^2) + \cdots + \psi(p^\ell)}{1+\ell} \\
		& \geq \prod_{p\leq y} \left(1+\frac{\ell}{\ell+1} \psi(p)\right)=\exp\left(\frac{\lambda \ell}{\ell+1} \log y +O(1)\right) \\
		& \geq (\log x)^{\lambda'} 
	\end{align*}
	for some $1<\lambda'<\lambda$ when $x$ is sufficiently large. We appeal again to Lemma~\ref{dyadic} to conclude that $\mathbf{T}_g$ is unbounded. 
\end{proof}
	We notice that, clearly, the symbol $g$  is not in $\BMOA(\mathbb{C}_0)$ for any $\lambda > 0$. In fact, for $\sigma > 0$,
	\[\sum_{n=1}^\infty \psi(n) n^{-\sigma} = \prod_{p} \left(1 - \psi(p)p^{-\sigma}\right)^{-1} \asymp \exp \Bigg(\lambda\sum_{p} \frac{\log{p}}{p^{1+\sigma}}\Bigg) \asymp e^{\lambda/\sigma}, \]
	which shows that $g$ is not even in the Smirnov class of  $\mathbb{C}_0$.

% HOMOGENEOUS
\section{Homogeneous symbols and coefficient estimates} \label{sec:homogen} The multiplicative symbols of the previous section represent analytic functions in $\mathbb{C}_0$. However, we saw in Theorem~\ref{thm:bmocond} that for $\mathbf{T}_g$ to be bounded, it is necessary that $g$ be  in $\BMOA(\mathbb{C}_{1/2})$. We will begin this section by showing that the latter condition cannot be relaxed by much. Indeed, to begin with, we will prove that linear Dirichlet series give examples of bounded $\mathbf{T}_g$-operators with symbols $g$ converging in $\mathbb{C}_{1/2}$ but in no larger half-plane. 
\begin{thm}
	\label{thm:linear} Let $g(s)=\sum_{p} b_p p^{-s}$ be any linear symbol in $\mathcal{H}^2$. Then $\|\mathbf{T}_g\| = \|g\|_{\mathcal{H}^2}$. 
\end{thm}
\begin{proof}
	We consider an arbitrary function $f(s)=\sum_{n\geq1} a_n n^{-s}$ in $\mathcal{H}^2$ and compute:
	\[\|\mathbf{T}_gf\|_{\mathcal{H}^2}^2 = \sum_{n=2}^\infty \frac{1}{(\log{n})^2}\Bigg|\sum_{p|n} b_p(\log p) a_{n/p} \Bigg|^2.\]
	By the Cauchy--Schwarz inequality
	\[ \Bigg|\sum_{p|n} b_p(\log p)a_{n/p} \Bigg|^2\leq \Bigg(\sum_{p|n}\log{p}\Bigg)\Bigg(\sum_{p|n} |b_p|^2 (\log p) |a_{n/p}|^2\Bigg)\leq (\log n) \sum_{p|n} |b_p|^2(\log{p})|a_{n/p}|^2.\]
	This shows that $\|\mathbf{T}_g\|\leq\|g\|_{\mathcal{H}^2}$. Since $\mathbf{T}_g 1 = g$, clearly $\|\mathbf{T}_g\|\geq\|g\|_{\mathcal{H}^2}$. 
\end{proof}

We note that the space of linear symbols $g$ in $\mathcal{H}^2$ is embedded not only in $\BMOA(\mathbb{C}_{1/2})$ but in fact satisfies the local Dirichlet integral condition
\[ \int_{0}^1\int_{1/2}^1 |g'(\sigma+it )|^2 d\sigma dt \ll \| g \|_2^2, \]
as shown in \cite[Example~4]{Olsen11}. We do not know if this stronger embedding can be established for a general symbol in $\mathcal{X}$.

While the norm of a linear function $g$ viewed as an element in the dual of $\mathcal{H}^1$ is also equivalent to $\|g\|_{\mathcal{H}^2}$ (see \cite{Helson06}), there is a striking contrast between the preceding result and the characterization of linear multipliers. Indeed, let again $M_g$ denote the operator of multiplication by $g$ on $\mathcal{H}^2$, and recall that $\|M_g\| = \|g\|_\infty$. (see \cite[Thm.~3.1]{HLS97}). Hence, in the special case when $g$ is linear, it follows from Kronecker's theorem that
\[\| M_g\|=\|g\|_\infty = \sup_{\sigma>0}\Bigg|\sum_p b_p p^{-s}\Bigg| = \sum_{p} |b_p|.\]
The difference between a linear symbol $g$ acting as a multiplier $M_g$ and as a symbol of the Volterra operator $\mathbf{T}_g$ is therefore dramatic: A bounded multiplier has coefficients in $\ell^1$, while the boundedness of $\mathbf{T}_g$ means that the coefficients are in $\ell^2$. The former implies absolute convergence in $\mathbb{C}_0$ and the latter only in $\mathbb{C}_{1/2}$. 

We may understand the phenomenon just observed in the following way. For a general symbol $g(s)=\sum_{n\ge 1} b_n n^{-s}$, we have, using also \eqref{eq:TgHi}, the series of inequalities
\begin{equation}\label{eq:ineqser} \left(\sum_{n=1}^\infty |b_n|^2\right)^{1/2} \le \| \mathbf{T}_g \| \le 2\| g\|_\infty\le 2\sum_{n=1}^{\infty} |b_n|. \end{equation}
The case of linear functions shows that neither the left nor the right inequality can be improved. Loosely speaking, the maximal independence between the terms in a linear symbol serves to make $\| \mathbf{T}_g \|$ minimal and thus equal to $\|g\|_2$ and, at the same time, to make $\| M_g \|$ maximal and hence equal to $\sum_{n\geq1} |b_n|$.  
This  motivates an investigation of what happens when the dependence between the terms in the symbol increases. Such a study, originating in the classical work of Bohnenblust and Hille \cite{BH31}, has already been made in the case of multipliers, in terms of $m$-homogeneous Dirichlet series. We will now follow the same path for $\mathbf{T}_g$-operators.  

Recall that $\Omega(n)$ gives the number of prime factors in $n$, counting multiplicities. An $m$-homogeneous Dirichlet series is of the form 
\begin{equation}
	\label{eq:mhomseri} g(s) :=\sum_{\Omega(n)=m} b_n n^{-s}.
\end{equation}
In this terminology, linear symbols are $1$-homogeneous Dirichlet series. A precise relationship between boundedness and absolute convergence for $m$-homogeneous Dirichlet series was found in \cite{BCQ06,MQ10}: \[\sum_{\Omega(n)=m} |b_n|\frac{(\log{n})^{\frac{m-1}{2}}}{n^\frac{m-1}{2m}} \leq C_m \Bigg\|\sum_{\Omega(n)=m} b_n n^{-s}\Bigg\|_\infty.\]
Here the exponent of $\log n$  on the left-hand side cannot be improved. Making the choice $m=\sqrt{\log{n}/\log_2{n}}$  in \eqref{eq:mhomseri}, we may obtain the following statement: If for some $c, C > 0$ we have
\begin{equation}
	\label{eq:sidon} \sum_{n=1}^\infty |b_n|\frac{\exp\left(c\sqrt{\log{n}\log_2{n}}\right)}{\sqrt{n}} \leq C \Bigg\|\sum_{n=1}^\infty b_n n^{-s}\Bigg\|_\infty,
\end{equation}
then $c < 1$. It was later shown in \cite{Breteche08,DFOOS11} that \eqref{eq:sidon} holds for $c < 1/\sqrt{2}$, and that this is optimal.

The series of inequalities \eqref{eq:ineqser} suggests that we should  search for upper $\ell^2$-estimates  for $\| \mathbf{T}_g\|$ as the appropriate analogues of the lower $\ell^1$ estimates \eqref{eq:mhomseri} and \eqref{eq:sidon}. Therefore, we now aim at finding weights $w_m(n)$ such that 
\begin{equation}
	\label{eq:homweight} \|\mathbf{T}_g\| \leq \Bigg(\sum_{\Omega(n)=m} |b_n|^2 w_m(n)\Bigg)^\frac{1}{2}
	\quad \text{for} \quad g(s)=\sum_{\Omega(n)=m} b_n n^{-s}.
\end{equation}
The crucial ingredient in the proof of Theorem~\ref{thm:linear} which covers the case $m=1$, is the estimate
\[\sum_{p|n} \log{p} \leq \log{n}.\]
To find a replacement for this estimate, we argue as follows. Observe that if $m\leq \Omega(n)$, then 
\begin{equation}
	\label{eq:divlogest} \sum_{k|n \atop \Omega(k)=m} \log{k} \leq \sum_{p_1|n}\sum_{p_2|n}\cdots \sum_{p_m|n} \log(p_1p_2\cdots p_m) = m \sum_{p_1|n} \cdots \sum_{p_m|n} \log p_m = m\omega(n)^{m-1}\log n. 
\end{equation}
This is sharp, up to a constant depending only on $m$. Indeed, let $n$ be square-free, so that $\Omega(n)=\omega(n)$. Then
\[\sum_{k|n \atop \Omega(k)=m} \log{k} = \sum_{p|n} \sum_{p|k|n \atop \omega(k)=m} \log{p} = \sum_{p|n}(\log{p})\binom{\omega(n)-1}{m-1}=(\log{n})\binom{\omega(n)-1}{m-1}.\]
This gives us an example of an admissible weight $w_2(n)$, since $\omega(n)/\log{n}$ is bounded. It turns out that we can obtain the following optimal result from \eqref{eq:divlogest}.
\begin{thm}
	\label{thm:homweights} The inequality in \eqref{eq:homweight} holds when $m=2$ with the weight function 
	\begin{equation}
		\label{eq:2hom} w_2(n) = C_2 \frac{\log{n}}{\log_2{n}}
	\end{equation}
	and $C_2$ an absolute constant. This is sharp in the sense that we cannot replace $\log_2{n}$ in \eqref{eq:2hom} by $(\log_2 n)^{1+\varepsilon}$ for any $\varepsilon>0$. When $m\geq3$,  the inequality in \eqref{eq:homweight} holds with 
	\begin{equation}
		\label{eq:mhom} w_m(n) = C_m \frac{n^{\frac{m-2}{m}}}{(\log{n})^{m-2}}
	\end{equation}
and $C_m$ an absolute constant.	This is also sharp in the sense that we cannot replace $(\log{n})^{m-2}$ in \eqref{eq:mhom} by $(\log n)^{m+\varepsilon-2}$ for any $\varepsilon>0$.
\end{thm}
\begin{proof}
	To prove that \eqref{eq:2hom} is sufficient, we let $\mathbf{T}_g$ act on $f(s)=\sum_{n\geq1} a_n n^{-s}$. By the Cauchy--Schwarz inequality, 
	\begin{align*}
		\|\mathbf{T}_gf\|_{\mathcal{H}^2}^2 &\leq \sum_{n=2}^\infty \frac{1}{(\log{n})^2}\Bigg(\sum_{k|n \atop \Omega(k)=2} (\log_2{k})\log{k}\Bigg)\Bigg(\sum_{k|n \atop \Omega(k)=2}|b_k|^2\frac{\log{k}}{\log_2{k}} |a_{n/k}|^2\Bigg) \\
		&\leq \sum_{n=2}^\infty \frac{\log_2{n}}{(\log{n})^2}\Bigg(\sum_{k|n \atop \Omega(k)=2} \log{k}\Bigg)\Bigg(\sum_{k|n \atop \Omega(k)=2}|b_k|^2\frac{\log{k}}{\log_2{k}} |a_{n/k}|^2\Bigg). 
	\end{align*}
	We complete the proof by using \eqref{eq:divlogest} and the well known estimate $\omega(n)\ll \log{n}/(\log_2{n})$. 
	
	To prove that \eqref{eq:2hom} is best possible, we assume that there is some $\varepsilon>0$ such that
	\[\|\mathbf{T}_g\| \leq C_2 \Bigg(\sum_{\Omega(n)=2}|b_n|^2 \frac{\log{n}}{(\log_2{n})^{1+\varepsilon}}\Bigg)^\frac{1}{2}\]
	for every $2$-homogeneous Dirichlet series $g$. Let $x$ be a large real number and consider the symbol
	\[g(s) = \sum_{x/2<p\leq x} \frac{\left(\log_2(pq)\right)^{1+\varepsilon/2}}{p}\,(pq)^{-s},\]
	where $q\sim e^x$ is a prime number. The weight condition is then satisfied uniformly in $x$, since
	\[\sum_{\Omega(n)=2} |b_n|^2 \frac{\log{n}}{(\log_2{n})^{1+\varepsilon}} = \sum_{x/2<p\leq x} \frac{\log(pq)\log_2(pq)}{p^2} \asymp \frac{x\log{x}}{x^2}\pi(x) \asymp 1.\]
	We now want to show that $\|\mathbf{T}_g\|$ is unbounded as $x \to \infty$, and choose as a test function 
	\begin{equation}
		\label{eq:ftest} f(s) := \prod_{x/2<p\leq x} \left(1+p^{-s}\right). 
	\end{equation}
	Let $S_x$ denote the set of square-free numbers generated by the primes $x/2<p\leq x$, so that $\|f\|_{\mathcal{H}^2}^2 = |S_x|=2^{N(x)}$, where $N(x):=\pi(x)-\pi(x/2)$. Note that if $n$ is in $S_x$, then $\omega(n)\leq N(x)$. It follows from the prime number theorem that
	\[N(x)\sim \frac{x}{2\log{x}}.\]
	Set $V_x := \left\{n \in S_x\,:\, \omega(n)\geq N(x)/2\right\}$. By the symmetry of the binomial expansion
	\[|S_x| = \sum_{n=0}^{N(x)} \binom{N(x)}{n} = \sum_{n < N(x)/2} \binom{N(x)}{2} + |V_x|,\]
	we find that $|V_x|\sim |S_x|/2$. Then 
	\begin{align*}
		\|\mathbf{T}_g\|^2 \geq \frac{\|\mathbf{T}_gf\|_{\mathcal{H}^2}^2}{\|f\|_{\mathcal{H}^2}^2} &\geq \frac{1}{|S_x|} \sum_{n \in V_x} \frac{1}{\left(\log(nq)\right)^2}\Bigg|\sum_{pq|nq} \frac{\left(\log_2(pq)\right)^{1+\varepsilon/2}}{p}\log(pq)\Bigg|^2 \\
		&\geq \frac{1}{|S_x|} \sum_{n \in V_x} \Bigg|\sum_{p|n} \frac{(\log_2{q})^{1+\varepsilon/2}}{p}\Bigg|^2 \asymp \frac{1}{|S_x|} \sum_{n \in V_x} \Bigg|\frac{(\log{x})^{1+\varepsilon/2}}{x}\,\omega(n)\Bigg|^2 &\asymp (\log{x})^\varepsilon, 
	\end{align*}
	giving the desired conclusion.
	
	The proof that \eqref{eq:mhom} is sharp is similar. Let $\varepsilon>0$  be given and consider
	\[g(s) = \sum_{n \in S_x \atop \omega(n)=m} n^{-1+1/m} (\log n)^{m-1+\varepsilon/2} n^{-s}.\]
	We observe that
	\[\sum_{\Omega(n)=m} |b_n|^2 n^{1-2/m}(\log{n})^{2-m-\varepsilon} = \sum_{n \in S_x \atop \omega(n)=m} \frac{(\log{n})^{m}}{n} \asymp \frac{(\log{x})^m}{x^m}(\pi(x))^m \asymp 1.\]
	Now, if $n$ is in $S_x$, then it follows from the prime number theorem that $\log{n}\ll x$. As test function, we use again \eqref{eq:ftest}. The function \[t \mapsto t^{-1+1/m}(\log{t})^{m-1+\varepsilon/2}\]
	is eventually decreasing for every $m\geq3$ and every $\varepsilon>0$. We find that 
	\begin{align*}
		\|\mathbf{T}_g\|^2 \geq \frac{\|\mathbf{T}_g f\|_{\mathcal{H}^2}^2}{\|f\|_{\mathcal{H}^2}^2} &\geq \frac{1}{|S_x|}\sum_{n \in V_x} \frac{1}{(\log{n})^2}\Bigg|\sum_{k|n \atop \Omega(k)=m} k^{-1+1/m}(\log{k})^{m-1+\varepsilon/2}\Bigg|^2 \\
		&\gg \frac{1}{|S_x|}\sum_{n\in V_x} \frac{1}{x^2} \Bigg|x^{-m+1}(\log{x})^{m+\varepsilon/2}\binom{\omega(n)}{m}\Bigg|^2 \gg (\log{x})^\varepsilon \frac{1}{|S_x|}\sum_{n \in V_x} 1 \gg (\log{x})^\varepsilon, 
	\end{align*}
	where we used that $k\leq x^{m}$ in the inner sum.
	
	It remains to establish that \eqref{eq:homweight} holds with the weight \eqref{eq:mhom}. Let $\mathbf{T}_g$ act on $f(s)=\sum_{n\geq1} a_n n^{-s}$. By the Cauchy--Schwarz inequality,
	\[\|\mathbf{T}_gf\|_{\mathcal{H}^2}^2 \leq \sum_{n=2}^\infty \frac{1}{(\log{n})^2}\Bigg(\sum_{k|n \atop \Omega(k)=m} k^{2/m-1}(\log k)^{m} \Bigg)\Bigg(\sum_{k|n \atop \Omega(k)=m}|b_k|^2 k^{1-2/m} (\log k)^{2-m} |a_{n/k}|^2\Bigg).\]
	Hence it suffices to show that
	\[A_m(n):=\sum_{k|n \atop \Omega(k)=m} k^{2/m-1} (\log k)^{m} \ll (\log n)^2.\]
	Suppose that $n$ has the prime factorization $n=(p_j)^\kappa$. Let $\widetilde{\kappa}$ denote a decreasing rearrangement of $\kappa$ and let $\widetilde{n}=(p_j)^{\widetilde{\kappa}}$. The function
	\[t \mapsto t^{2/m-1}(\log t)^m\]
	is eventually decreasing for every $m\geq3$, so clearly $A_m(n) \ll A_m(\widetilde{n})$. On the other hand $\widetilde{n} \leq n$, so we may without loss of generality assume that $n = \widetilde{n}$. Hence, we have that
	\[n = p_1^{\kappa_1}\cdots p_d^{\kappa_d},\]
	where $\kappa_1 \geq \kappa_2 \geq \cdots \geq \kappa_d>0$. By the prime number theorem, 
	\begin{equation}
		\label{eq:two} p_d \sim \sum_{p \leq p_d} \log{p} = \log\Bigg(\prod_{j=1}^d p_j\Bigg) \leq \log n. 
	\end{equation}
	By summing over the largest prime first, we find that
	\[A_m(n)\leq \sum_{p\leq p_d} (m\log p)^{m} p^{2/m-1} \Bigg(\sum_{q\leq p} q^{2/m-1}\Bigg)^{m-1} \ll \sum_{p\leq p_d} p(\log p) \ll p_d^2 \ll (\log{n})^2\]
	using the prime number theorem twice. 
\end{proof}

As promised, Theorem~\ref{thm:homweights} exhibits $m$-homogeneous Dirichlet series $g$ in $\mathcal{X}$ that converge in $\mathbb{C}_{1/m}$, but in no larger half-plane, for every $m\geq2$. This can be loosely interpreted as saying that the more prime factors we have in each non-zero term, the closer we get to the half-plane $\mathbb{C}_0$. In this sense, the multiplicative symbols of Section~\ref{sec:multip} correspond to $m=\infty$, and it is therefore not surprising that they converge in $\mathbb{C}_0$. 

Setting $m=\sqrt{2\log{n}/\log_2{n}}$, we are led to a family of weights $w$ (cf. \eqref{eq:sidon}) that give estimates of the type \eqref{eq:homweight} with no reference to homogeneity, allowing arbitrary Dirichlet series $g$. 
%In this case, convergence in all of $\mathbb{C}_0$. 
\begin{thm}
	\label{thm:genweight} If $c < 2$, then 
	\begin{equation}
		\label{eq:genweight} \|\mathbf{T}_g\|\leq C\Bigg(\sum_{n=2}^\infty |b_n|^2 n \exp\left(-c\sqrt{\log{n}\log_2{n}}\right)\Bigg)^\frac{1}{2}. 
	\end{equation}
	Conversely, if \eqref{eq:genweight} holds for every $\mathbf{T}_g$-operator, then $c \le 2\sqrt{2}$. 
\end{thm}
\begin{proof}
	 We observe first that we must have $c \le 2\sqrt{2}$ for \eqref{eq:genweight} to hold in view of the sharpness of   Theorem~\ref{thm:homweights} and the fact that
	\[\frac{n^{1-2/m}}{(\log{n})^{m-2}} = n(\log{n})^2 \exp\left(-2\sqrt{2}\sqrt{\log{n}\log_2{n}}\right),\]
	if $m=\sqrt{2\log{n}/\log_2{n}}$. 
	
	It remains therefore only to show the positive result that \eqref{eq:genweight} holds whenever $0 < c < 2$. To simplify the notation, we set $\varphi_c(k) := \exp\left(c\sqrt{\log{k}\log_2{k}}\right)$. By the Cauchy--Schwarz inequality,
	\[\|\mathbf{T}_gf\|_{\mathcal{H}^2}^2 \leq \sum_{n=2}^\infty \frac{1}{(\log{n})^2}\Bigg(\sum_{k|n} \frac{\varphi_c(k)}{k}(\log{k})^2\Bigg)\Bigg(\sum_{k|n}|b_k|^2\frac{k}{\varphi_c(k)}|a_{n/k}|^2\Bigg).\]
	Choosing some $c'$, $c<c'<2$, we find that
	\[\sum_{k|n}\frac{\varphi_c(k)}{k}(\log{k})^2 \ll \sum_{k|n} \frac{\varphi_{c'}(k)}{k} =: A(n)\]
	The rest of the proof is devoted to showing that $A(n) \ll (\log n)^2$, which is precisely what is needed. 
	
	% Assuming $c'$ is fixed, we drop the subscript $c'$ so that $\varphi(n):=\varphi_{c'}(n)$.  
	Since $x \mapsto \varphi_{c'}(x)/x$ is eventually decreasing on $[1,\infty)$, we may, as in the last part of the proof of Theorem~\ref{thm:homweights}, assume that $n=\widetilde{n}$. By splitting into homogeneous parts and using \eqref{eq:two}, we find that
	\[A(n) = \sum_{m\leq\Omega(n)} \sum_{k|n \atop \Omega(k)=m} \frac{\varphi_{c'}(k)}{k} \leq \sum_{m\leq\Omega(n)} \varphi_{c'}\left((\log{n})^m\right) \sum_{k|n \atop \Omega(k)=m} \frac{1}{k}.\]
	In each inner sum $\sum k^{-1}$, we divide every prime factor of $k$ by some $a > 0$ and then bound the resulting sum by an Euler product (Rankin's trick), to obtain that 
	\begin{align*}
		\sum_{k|n \atop \Omega(k)=m} \frac{1}{k} &\leq a^{-m} \prod_{p|n} \left(1-\frac{a}{p}\right)^{-1} = a^{-m}\exp\Bigg(a\sum_{p|n} \frac{1}{p} + O(1)\Bigg) \\
		&\ll a^{-m} \exp\Bigg(a\sum_{p \leq p_d} \frac{1}{p}\Bigg) \asymp a^{-m}\exp\left(a\log_2 p_d\right) \leq \exp\left(-m\log{a}+a\log_3{n}\right). 
	\end{align*}
	Choosing $a:=m/(\log_3{n})$, we obtain in total
	\begin{align*}
		A(n) &\leq \sum_{m\leq \Omega(n)} \exp\left[c'\sqrt{m(\log_2 n)(\log m +\log_3 n)} -m\log m+m\log_4 n+m\right] \\
		&\ll \Omega(n) + \sum_{m\leq \log_2{n}} \exp\left[c'\sqrt{2m(\log_2 n)(\log_3 n)} -m\log m+m\log_4 n+m\right],
	\end{align*}
	where we first used that the exponential in the sum is bounded when $m\geq\log_2{n}$, and then that $\log{m}\leq\log_3{n}$ when $m\leq \log_2{n}$. To estimate the final sum, we use calculus to conclude that the index $m$ of the largest term should satisfy
	\[\frac{{c'}^2}{2}(\log_2{n})(\log_3{n}) = m\left(\log{m}-\log_4{n}\right)^2,\]
	and we see that $m=({c'}^2/2+o(1))\log_2{n}/\log_3{n}$. Combining this with the standard estimate $\Omega(n)\leq \log{n}/\log{2}$, we find that
	\[A(n) \ll \log{n} + (\log_2{n}) \exp\left(\left(\frac{{c'}^2}{2}+o(1)\right)(\log_2 n)\right) \ll (\log n)^2,\]
	whenever ${c'}^2<4$, which is the desired estimate.
	%\[A(n) \leq \sum_{m\leq \Omega(n)} \exp\left[c'\sqrt{m(\log_2 n)(\log m +\log_3 n)} -m\log m+m\log_4 n+m\right].\]
	%The terms in this sum decay exponentially if, say $m\log m> 2 \log_2 n$, so it suffices to consider the 
	%those $m$ for which $m\log m=\xi  \log_2 n$ for $\xi\le 2$. But then
	%\[ \sqrt{m(\log_2 n)(\log m +\log_3 n)} = (\sqrt{2\xi}+o(1)) \log_2 n, \]
	%so that, maximizing with respect to $\xi$ and taking into account that we only need to sum up to $2\log_2 n$ terms, we find that
	%\[A(n)\ll (\log_2 n) \exp((c^2/2+o(1))\log_2 n) \ll (\log n)^2, \]
	%which is the desired estimate.
\end{proof}
	It is not surprising that there is a gap between the necessary and sufficient conditions of Theorem~\ref{thm:genweight}. When considering the inequality \eqref{eq:sidon}, the necessary condition obtained from $m$-homogeneous Dirichlet series misses the sharp condition, also by a factor $\sqrt{2}$. In the latter case, the proof of the sharp necessary condition captures cancellations by $L^\infty$ estimates for random trigonometric polynomials \cite{Breteche08}. This suggests that our arguments, which only deal with the absolute values of the coefficients of $g$, cannot be expected to tell the full story. 

% Hp
\section{Boundedness of $\mathbf{T}_g$ on $\mathcal{H}^p$}\label{sec:embed} 
\subsection{Carleson measure characterization} We will now consider the action of the Volterra operator $\mathbf{T}_g$ on the Hardy spaces $\mathcal{H}^p$, for $0<p<\infty$. To this end, we recall that $\mathcal{X}_p$ denotes the space of symbols $g$ in $\mathcal{D}$ such that the Volterra operator $\mathbf{T}_g$ acts boundedly on $\mathcal{H}^p$, and we set
\[\|g\|_{\mathcal{X}_p} := \|\mathbf{T}_g\|_{\mathcal{L}(\mathcal{H}^p)}.\]
We will now establish our characterization of the elements of $\mathcal{X}_p$ in terms of Carleson measures. 

Applying the Littlewood--Paley formula \eqref{eq:LPp} to $\mathbf{T}_gf$, we immediately obtain a characterization of the symbols $g$ that belong to $\mathcal{X}_2$: $g$ is in $\mathcal{X}_2$ if and only if it there is a positive constant $C(g)$ such that
\[\|\mathbf{T}_g f\|_{\mathcal{H}^2}^2 \asymp \int_{\mathbb{T}^\infty} \int_\mathbb{R} \int_0^\infty |f_\chi(\sigma+it)|^2 |g'_\chi(\sigma+it)|^2 \sigma\,d\sigma\,\frac{dt}{1+t^2}\,dm_\infty(\chi) \leq C(g)^2 \|f\|_{\mathcal{H}^2}^2.\]
Using Fubini's theorem, we may remove the integral over $\mathbb{R}$, since each $t$ represents a rotation in each variable on $\mathbb{T}^\infty$. From this observation we obtain the characterization 
\begin{equation}
	\label{eq:CD2} \int_{\mathbb{T}^\infty}\int_0^\infty |f_\chi(\sigma)|^2\,|g_\chi'(\sigma)|^2 \,\sigma d\sigma\,dm_\infty(\chi) \leq C(g)^2\|f\|_{\mathcal{H}^2}^2. 
\end{equation}
Clearly, the smallest constant $C(g)$ in \eqref{eq:CD2} satisfies $C(g) \asymp \|\mathbf{T}_g\|_{\mathcal{L}(\mathcal{H}^2)}$. \begin{thm}
	\label{thm:avcarl} $\mathbf{T}_g$ acts boundedly on $\mathcal{H}^p$ for $0<p<\infty$ if and only if there is a positive constant $C(g,p)$ such that 
	\begin{equation}
		\label{eq:tempo} \int_{\mathbb{T}^\infty}\int_0^\infty |f_\chi(\sigma)|^p |g'_\chi(\sigma)|^2\sigma\,d\sigma dm_\infty(\chi) \leq C(g,p)^2 \|f\|_{\mathcal{H}^p}^p, 
	\end{equation}
	for all $f\in\mathcal{H}^p$. Furthermore, if 
	\begin{equation}
		\label{eq:cpgdef} C(g,p) := \sup_{\|f\|_{\mathcal{H}^p}=1}\left(\int_{\mathbb{T}^\infty}\int_0^\infty |f_\chi(\sigma)|^p |g'_\chi(\sigma)|^2\sigma\,d\sigma dm_\infty(\chi)\right)^\frac{1}{2}, 
	\end{equation}
	then $C(g,p)\asymp \|\mathbf{T}_g\|_{\mathcal{L}(\mathcal{H}^p)}$. 
\end{thm}
We observe that if we restrict to only one variable, meaning that we consider only Dirichlet series over powers of a single prime, then the condition of Theorem~\ref{thm:avcarl} is independent of $p$ and reduces to the familiar one variable description of $\BMOA(\mathbb{D})$. 

Our proof of Theorem~\ref{thm:avcarl} adapts arguments from \cite{Pau13}, the main difference being that we will additionally integrate every quantity over $\mathbb{T}^\infty$. Before giving the proof, we collect some preliminary results. By using Fubini's theorem once more, we find that \eqref{eq:tempo} is equivalent to 
\begin{equation}
	\label{eq:avcarlesonp} \int_{\mathbb{T}^\infty} \int_\mathbb{R} \int_0^\infty |f_\chi(\sigma+it)|^p |g'_\chi(\sigma+it)|^2 \sigma\,d\sigma\,\frac{dt}{1+t^2}\,dm_\infty(\chi) \leq C(g,p)^2 \|f\|_{\mathcal{H}^p}^p. 
\end{equation}
The virtue of introducing an extra parameter in \eqref{eq:avcarlesonp} is that it allows us to apply techniques adapted to the conformally invariant Hardy space $H^p_{\operatorname{i}}(\mathbb{C}_0)$. In addition to the Littlewood--Paley formula \eqref{eq:LPp}, we will use the square function formula 
\begin{equation}
	\label{eq:sqfcn} \|f\|_{\mathcal{H}^p}^p \asymp |a_1|^p + \int_{\mathbb{T}^\infty} \int_{\mathbb{R}}\left(\int_{\Gamma_\tau} |f_\chi'( \sigma + it)|^2 \, d\sigma \, dt \right)^{p/2} \, \frac{d\tau}{1+\tau^2} \,dm_\infty(\chi), 
\end{equation}
which can be found in \cite[Thm.~7]{BP15}. Here, for $\tau$ in $\mathbb{R}$, $\Gamma_\tau$ is the cone 
\begin{equation*}
	\Gamma_\tau = \{ \sigma + i t \, : \, |t - \tau | < \sigma\}. 
\end{equation*}
For a holomorphic function $f$ in $\mathbb{C}_0$, let $f^\ast$ denote the non-tangential maximal function 
\begin{equation}
	\label{eq:maxfcn} f^\ast(\tau) := \sup_{s \in \Gamma_\tau} |f(s)|, \qquad \tau \in \mathbb{R}. 
\end{equation}
Since $1/(1+\tau^2)$ is a Muckenhoupt $A_q$-weight for all $q > 1$, it follows from the work of Gundy and Wheeden \cite{GW7374} that $f$ is in $H^p_\mathrm{i}(\mathbb{C}_0)$ if and only if $f^\ast$ is in $L^p_\mathrm{i}(\mathbb{R}) = L^p \left((1+\tau^2)^{-1} \, d\tau\right)$ for $0 < p < \infty$, with comparable norms. 
\begin{lem}
	\label{lem:cones} Let $\varphi$ be a function and $\mu$ a positive measure on $\{\sigma + it \, : \, 0 < \sigma < 1\}$. Then 
	\begin{equation}
		\label{eq:convert} \int_{\mathbb{R}} \int_0^1 |\varphi(\sigma+it)| d \mu(\sigma, t) \asymp \int_{\mathbb{R}} \int_{\Gamma_\tau} |\varphi(\sigma+it)| \frac{1+t^2}{\sigma} d\mu(\sigma,t) \, \frac{d\tau}{1+\tau^2}. 
	\end{equation}
	If $\mu$ is a positive measure on all of\/ $\mathbb{C}_0$, then 
	\begin{equation}
		\label{eq:convert2} \int_{\mathbb{R}} \int_0^\infty |\varphi(\sigma+it)| d \mu(\sigma, t) \gg \int_{\mathbb{R}} \int_{\Gamma_\tau} |\varphi(\sigma+it)| \frac{1+t^2}{\sigma} d\mu(\sigma,t) \, \frac{d\tau}{1+\tau^2}. 
	\end{equation}
\end{lem}
\begin{proof}
	For $\sigma + it$ in $\mathbb{C}_0$, we consider the set $I(\sigma + it) := \{\tau \in \mathbb{R} \, : \, \sigma + it \in \Gamma_\tau\}$. A computation shows that
	\[\int_{I(\sigma + it)} \frac{d \tau}{1+\tau^2} \asymp \frac{\sigma}{1+t^2}, \qquad 0 < \sigma \leq 1\]
	and that
	\[\int_{I(\sigma + it)} \frac{d \tau}{1+\tau^2} \ll \frac{\sigma}{1+t^2}, \qquad 0 < \sigma < \infty.\]
	The estimates \eqref{eq:convert} and \eqref{eq:convert2} now follow from Fubini's theorem. 
\end{proof}
\begin{proof}
	[Proof of Theorem~\ref{thm:avcarl}] We may assume that $g$ is in $\mathcal{H}^p$ since otherwise $\mathbf{T}_g$ is trivially unbounded. Thus, for almost every $\chi$ in $\mathbb{T}^\infty$, the measure
	\[\mu_{g,\chi}(\sigma,t) = |g'_\chi(\sigma+it)|^2\,\sigma d\sigma\,\frac{dt}{1+t^2}\]
	is well-defined on $\mathbb{C}_0$.
	
	Suppose first that $p \geq 2$ and that \eqref{eq:avcarlesonp} is satisfied. Then by the Littlewood--Paley formula \eqref{eq:LPp}, H\"{o}lder's inequality, and two applications of \eqref{eq:avcarlesonp}, we have that 
	\begin{align*}
		\|\mathbf{T}_g f\|_{\mathcal{H}^p}^p &\asymp \int_{\mathbb{T}^\infty} \int_{\mathbb{R}}\int_0^\infty |(\mathbf{T}_g f)_\chi(\sigma+it)|^{p-2}|f_\chi(\sigma+it)|^2 d \mu_{g,\chi}(\sigma, t) \, dm_\infty(\chi) \\
		&\leq \left( \int_{\mathbb{T}^\infty} \int_{\mathbb{R}}\int_0^\infty |(\mathbf{T}_g f)_\chi|^{p} d \mu_{g,\chi}(\sigma,t) \, dm_\infty(\chi) \right)^{\frac{p-2}{p}} \left( \int_{\mathbb{T}^\infty} \int_{\mathbb{R}}\int_0^\infty |f_\chi|^{p} d \mu_{g,\chi}(\sigma,t) \, dm_\infty(\chi) \right)^{\frac{2}{p}} \\
		&\ll C(g,p)^2 \|\mathbf{T}_g f\|_{\mathcal{H}^p}^{p-2} \|f\|_{\mathcal{H}^p}^2, 
	\end{align*}
	giving us that $\|\mathbf{T}_g f\|_{\mathcal{H}^p} \ll C(g,p) \|f\|_{\mathcal{H}^p}.$
	
	Suppose now that $\mathbf{T}_g$ acts boundedly on $\mathcal{H}^p$, still considering $p \geq 2$. By \eqref{eq:convert}, H\"older's inequality, \eqref{eq:maxfcn}, \eqref{eq:avgrotemb}, and the square function characterization, we have 
	\begin{align*}
		\int_{\mathbb{T}^\infty} \int_\mathbb{R} \int_0^1 |f_\chi|^p d\mu_{g,\chi} \,dm_\infty &\ll \int_{\mathbb{T}^\infty} \int_{\mathbb{R}} \int_{\Gamma_\tau} |f_\chi(\sigma + it)|^p |g'_\chi(\sigma + it)|^2 d\sigma dt \frac{d \tau}{1+\tau^2} dm_\infty(\chi) \\
		&\leq \int_{\mathbb{T}^\infty} \int_{\mathbb{R}} \left(f^\ast_\chi(\tau)\right)^{p-2} \int_{\Gamma_\tau} |(\mathbf{T}_g f)_\chi'|^2 d \sigma dt \frac{d \tau}{1+\tau^2} dm_\infty(\chi) \\
		&\leq \|f\|_{\mathcal{H}^p}^{p-2} \|\mathbf{T}_g f\|_{\mathcal{H}^p}^2 \ll \|\mathbf{T}_g\|_{\mathcal{L}(\mathcal{H}^p)}^2 \|f\|_{\mathcal{H}^p}^p. 
	\end{align*}
	The remaining integral can be estimated using the uniform pointwise estimates that hold for $f$ and $g$ in $\mathcal{H}^p$ in the half-plane $\mre(s)\geq1$, yielding that
	\[\int_{\mathbb{T}^\infty} \int_\mathbb{R} \int_1^\infty |f_\chi|^p d\mu_{g,\chi}(\sigma,t) \,dm_\infty(\chi) \ll \|f\|_{\mathcal{H}^p}^p\|g\|_{\mathcal{H}^p}^{2} \leq \|f\|_{\mathcal{H}^p}^p\|\mathbf{T}_g\|_{\mathcal{L}(\mathcal{H}^p)}^{2}.\]
	
	Suppose now that $0 < p < 2$ and that \eqref{eq:avcarlesonp} is satisfied. Using the square function characterization \eqref{eq:sqfcn}, \eqref{eq:maxfcn}, H\"older's inequality, \eqref{eq:avgrotemb}, and \eqref{eq:convert2}, we obtain 
	\begin{align*}
		\|\mathbf{T}_g f\|_{\mathcal{H}^p}^p &\asymp \int_{\mathbb{T}^\infty} \int_{\mathbb{R}}\left(\int_{\Gamma_\tau} |f_\chi(\sigma +it)|^2|g'_\chi(\sigma + it)|^2 \, d\sigma \, dt \right)^{p/2} \, \frac{d\tau}{1+\tau^2} \,dm_\infty(\chi) \\
		&\leq \int_{\mathbb{T}^\infty} \int_{\mathbb{R}} \left(f^\ast_\chi(\tau)\right)^{\frac{(2-p)p}{2}}\left(\int_{\Gamma_\tau} |f_\chi(\sigma +it)|^p|g'_\chi(\sigma + it)|^2 \, d\sigma \, dt \right)^{p/2} \, \frac{d\tau}{1+\tau^2} \,dm_\infty(\chi) \\
		&\leq \|f\|_{\mathcal{H}^p}^{\frac{(2-p)p}{2}} \left( \int_{\mathbb{T}^\infty} \int_{\mathbb{R}} \int_{\Gamma_\tau} |f_\chi(\sigma +it)|^p|g'_\chi(\sigma + it)|^2 \, d\sigma \, dt \, \frac{d\tau}{1+\tau^2} \,dm_\infty(\chi) \right)^{\frac{p}{2}} \\
		&\ll \|f\|_{\mathcal{H}^p}^{\frac{(2-p)p}{2}} \left( \int_{\mathbb{T}^\infty} \int_{\mathbb{R}} \int_0^\infty |f_\chi(\sigma +it)|^p d\mu_{g,\chi}(\sigma,t) \, \,dm_\infty(\chi) \right)^{\frac{p}{2}} \\
		&\leq \|f\|_{\mathcal{H}^p}^{\frac{(2-p)p}{2}} C(g,p)^p \|f\|_{\mathcal{H}^p}^{\frac{p^2}{2}} = C(g,p)^p \|f\|_{\mathcal{H}^p}^{p}. 
	\end{align*}
	
	Finally we deal with the case when $0 < p < 2$ and $\mathbf{T}_g \colon \mathcal{H}^p \to \mathcal{H}^p$ is bounded. Note first that by the Littlewood--Paley formula \eqref{eq:LPp}, we have
	\[ \|\mathbf{T}_g f\|_{\mathcal{H}^p}^p \asymp \int_{\mathbb{T}^\infty} \int_\mathbb{R} \int_0^\infty |(\mathbf{T}_g)_\chi|^{p-2} |f_\chi|^2 d\mu_{g,\chi}(\sigma,t) \,dm_\infty(\chi).\]
	Using H\"older's inequality and this identity, we obtain 
	\begin{align*}
		\int_{\mathbb{T}^\infty} \int_\mathbb{R} \int_0^\infty |f_\chi|^p d\mu_{g,\chi}(\sigma,t) \,dm_\infty(\chi) &\ll \|\mathbf{T}_g f\|_{\mathcal{H}^p}^{\frac{p^2}{2}} \left( \int_{\mathbb{T}^\infty} \int_\mathbb{R} \int_0^\infty |(\mathbf{T}_gf)_\chi|^{p} d\mu_{g,\chi}(\sigma,t) \,dm_\infty(\chi)\right)^{\frac{2-p}{2}} \\
		&\leq \|\mathbf{T}_g f\|_{\mathcal{H}^p}^{\frac{p^2}{2}} C(g,p)^{2-p} \|\mathbf{T}_g f\|_{\mathcal{H}^p}^{\frac{p(2-p)}{2}} \\
		&\leq C(g,p)^{2-p} \|\mathbf{T}_g\|_{\mathcal{L}(\mathcal{H}^p)}^p \|f\|_{\mathcal{H}^p}^p. 
	\end{align*}
	By an approximation argument, we can a priori assume that $C(g,p)$ is finite. Then, by taking the supremum over norm-$1$ Dirichlet series $f$, we obtain that $C(g,p) \ll \|\mathbf{T}_g\|_{\mathcal{L}(\mathcal{H}^p)},$ as desired. 
\end{proof}

\subsection{Necessary and sufficient conditions} \label{sec:embedsuff}
Theorem~\ref{thm:avcarl} can be applied to find necessary and sufficient conditions for membership in $\mathcal{X}_p$, parallel to the result for $\mathcal{X}_2$ proved in Theorem~\ref{thm:bmocond}. However, there is one essential difficulty when passing from $p=2$ to the general case $0<p<\infty$, namely that the proof of part (c) of Theorem~\ref{thm:bmocond} relies on the local embedding property of $\mathcal{H}^2$ expressed by \eqref{eq:localiemb}. The local embedding extends trivially to hold for $p=2k$, for every positive integer $k$, since 
\begin{equation}
	\label{eq:emb2k} \|f\|_{H^{2k}_{\operatorname{i}}(\mathbb{C}_{1/2})}^{2k} = \big\|f^k\big\|_{H^{2}_{\operatorname{i}}(\mathbb{C}_{1/2})}^{2} \leq \widetilde{C}\big\|f^k\big\|_{\mathcal{H}^{2}}^{2} = \widetilde{C} \|f\|_{\mathcal{H}^{2k}}^{2k}, 
\end{equation}
but it is a well-known open problem whether it holds for any other $p$. We refer to \cite[Sec.~3]{SS09} for a discussion of the embedding problem.

Arguing similarly for the embedding constant \eqref{eq:cpgdef}, we find for every positive integer $n$ that 
\begin{equation}
	\label{eq:cpgn} C(g,p)\geq C(g,np). 
\end{equation}
We will use this to prove a rather curious incomplete analogue to part (c) of Theorem~\ref{thm:bmocond}. In view of \eqref{eq:emb2k} and \eqref{eq:cpgn}, we are allowed to apply integral powers before and after using the local embedding property of $\mathcal{H}^2$, leading us to the expected necessary condition for $g$ to belong to $\mathcal{X}_p$, but only for rational $p$.
\begin{thm}
	\label{thm:bmoxp} Suppose that $g$ is in $\mathcal{D}$.
	\begin{itemize}
		\item[(a)] If $g$ is in $\BMOA(\mathbb{C}_0)$, then $\mathbf{T}_g$ is bounded from $\mathcal{H}^p$ to $\mathcal{H}^p$. 
		\item[(b)] If $g$ is in $\mathcal{X}_p$, then $g$ satisfies condition {\normalfont (iii)} from Lemma~\ref{lem:chibmo}. 
		\item[(c)] If $g$ is in $\mathcal{X}_p$ and $p$ is in $\mathbb{Q}_+$, then $g$ is in $\BMOA(\mathbb{C}_{1/2})$. 
	\end{itemize}
\end{thm}
\begin{proof}
	The proof of (a) is identical to the proof given for $p=2$ in Theorem~\ref{thm:bmocond}, using Theorem~\ref{thm:avcarl}, \eqref{eq:avcarlesonp}, and that Carleson measures in one variable are independent of $p$. The proof of (b) is also the same.
	
	For (c) we need two facts which follow from close inspection of the proof of Theorem~\ref{thm:avcarl}. First of all, it is clear from the first part of the proof that for $p \geq 2$ there is a constant $C_1$, independent of $p$, such that
	\[ \|\mathbf{T}_g\|_{\mathcal{L}(\mathcal{H}^p)} \leq C_1 C(g,p), \]
	where $C(g,p)$ is as in Theorem~\ref{thm:avcarl}. Hence, we conclude by \eqref{eq:cpgn} that there is a constant $C_2$ such that for every positive integer $n$ we have 
	\begin{equation}
		\label{eq:tgrise} \|\mathbf{T}_g\|_{\mathcal{L}(\mathcal{H}^{np})} \leq C_2 \|\mathbf{T}_g\|_{\mathcal{L}(\mathcal{H}^{p})}. 
	\end{equation}
	Secondly, by mimicking the next part of the proof, also for $p \geq 2$, we see that there is a constant $C_3$ such that 
	\begin{equation}
		\label{eq:carlbound} \int_{\mathbb{R}}\int_{1/2}^1 |f(s)|^p|g'(s)|^2(\sigma-1/2)d\sigma\frac{dt}{1+t^2} \leq C_3 \|\mathbf{T}_g f\|^2_{H^p_{\operatorname{i}}(\mathbb{C}_{1/2})} \| f\|^{p-2}_{H^p_{\operatorname{i}}(\mathbb{C}_{1/2})}, 
	\end{equation}
	at least for Dirichlet polynomials $f$. Here we have implicitly applied the maximal function characterization of $H^p_{\operatorname{i}}(\mathbb{C}_{1/2})$. However, by the inner-outer factorization of $H^p_{\operatorname{i}}$, we see that the constants involved do not blow up as $p \to \infty$.
	To prove the theorem, let $p = 2k/n > 0$ be a rational number. Hence, by \eqref{eq:tgrise}, $\mathbf{T}_g$ is bounded on $\mathcal{H}^{2k}$, with control of the constant. Combined with \eqref{eq:carlbound} and the embedding \eqref{eq:emb2k}, we find, setting $C_4=\widetilde{C}$, that 
	\begin{align*}
		\int_{\mathbb{R}}\int_{1/2}^1 |f(s)|^{2k}|g'(s)|^2(\sigma-1/2)d\sigma\frac{dt}{1+t^2} &\leq C_3 \|\mathbf{T}_g f\|^2_{H^{2k}_{\operatorname{i}}(\mathbb{C}_{1/2})} \| f\|^{2(k-1)}_{H^{2k}_{\operatorname{i}}(\mathbb{C}_{1/2})} \\
		&\leq C_3 C_4^2 C_2^2 \|\mathbf{T}_g\|_{\mathcal{L}(\mathcal{H}^p)}^2 \|f\|_{\mathcal{H}^{2k}}^{2k}. 
	\end{align*}
	It follows that $\nu_g(\sigma+it) := |g'(s)|^2 (\sigma - 1/2) \, d\sigma dt/(1+t^2)$ is a Carleson measure for $\mathcal{H}^{2k}$, with constant uniformly bounded by $\|\mathbf{T}_g\|_{\mathcal{L}(\mathcal{H}^p)}^2$. Clearly, the argument in \cite[Thm.~3]{OS12} produces uniform estimates, so we conclude that $\nu_g$ is a Carleson measure on $H^{2k}_{\operatorname{i}}(\mathbb{C}_{1/2})$, with constant uniformly bounded by the same quantity. By appealing to the inner-outer factorization again, we conclude that there is a constant $C_5$ such that
	\[ \|\nu_g\|_{\mathrm{CM}(H^{2}_{\operatorname{i}})} \leq C_5 \|\mathbf{T}_g\|_{\mathcal{L}(\mathcal{H}^p)}^2 \leq C_6 C(g,p)^2. \]
	The proof is now completed by arguing as at the end of the proof of Theorem~\ref{thm:bmocond}. 
\end{proof}
Theorem~\ref{thm:zetabmo} now gives us an interesting example of a $\mathbf{T}_g$-operator that is bounded on all $\mathcal{H}^p$-spaces. 
\begin{cor}
	Let $g$ be as in  Theorem~\ref{thm:zetabmo}, i.e.,
	\[g(s) = \sum_{n=2}^\infty \frac{1}{n \log n}n^{-s}.\]
	Then $\mathbf{T}_g : \mathcal{H}^p \to \mathcal{H}^p$ is bounded for every $p < \infty$. 
\end{cor}

\subsection{Linear symbols} \label{sec:embedlinear}
We will now extend Theorem~\ref{thm:linear} by proving that all linear symbols $g$ yield bounded operators $\mathbf{T}_g$ on $\mathcal{H}^p$, for the whole range $0<p<\infty$. We do this by showing that in this special case, the constant $C(g,p)$ in the Carleson measure condition \eqref{eq:tempo} may be chosen independently of $p$. 
\begin{thm}\label{thmlin}
Let 
\[g(s)=\sum_{j=1}^\infty b_j p_j^{-s}\]
be given. Then $\mathbf{T}_g$ is bounded on $\mathcal{H}^p$ if and only if $g$ is in $\mathcal{H}^2$. In fact, 
\[ \sup_{f\in \mathcal{H}^p, \| f\|_{\mathcal{H}^p}\le 1} \int_{\mathbb{T}^{\infty}} \int_{0}^\infty | f_{\chi}(\sigma)|^p |g'_{\chi} (\sigma)|^2 \sigma d\sigma dm_{\infty}(\chi)= \frac{1}{4} \| g\|_{\mathcal{H}^2}^2  \]
holds whenever $0<p<\infty$.
\end{thm}
It suffices to consider finitely many, say $d$, variables. The Poisson kernel on the polydisc is then given by
\[ P_{z}(w):=\prod_{j=1}^d \frac{1-|z_j|^2}{|1-\overline{w_j} z_j|^2}, \]
where $|z_j|<1$ and $w=(w_j)$ is a point on $\mathbb{T}^{d}$. Suppose that $0 < \alpha \leq p$ and that $f$ is in $H^p(\mathbb{D}^d)$.  Then $|f|^\alpha$ is separately subharmonic in each variable, which gives us the following.
\begin{lem}\label{lempoint}
If $f$ is in $H^p(\mathbb{D}^d)$, then 
\[ |f(z)|^{\alpha} \le \int_{\mathbb{T}^d} P_z(w) |f(w)|^{\alpha} dm_d(w) \]
for every point $z$ in $\mathbb{D}^d$ and $0<\alpha\le p$.
\end{lem}

Lemma~\ref{lempoint} shows that if the Carleson embedding condition \eqref{eq:tempo} holds for all harmonic functions $f$, for one $p$, then \eqref{eq:tempo} holds for all $f$ in $\mathcal{H}^p$, for every $p$. Hence, to prove Theorem~\ref{thmlin}, we only need to verify that linear functions $g$ in $\mathcal{H}^2$ induce Carleson measures on the harmonic functions for $p=2$. Obviously this raises the question whether the corresponding statement is true for other symbols $g$ from Sections~\ref{sec:multip} and \ref{sec:homogen}, or even if it could be true that the Carleson condition for analytic functions implies the same condition for harmonic functions, cf. Question~1 in the introduction. We only have the answer in the simplest case of linear symbols. 
	
To simplify the computations to be given below, we will use the multiplicative notation that comes from identifying the dual of the compact abelian group $\mathbb{T}^\infty$ with the discrete abelian group $\mathbb{Q}_+$ (see \cite{HLS97,QQ13}). This means that the Fourier series of $f$ on $\mathbb{T}^\infty$ takes the form
\[ \sum_{r \in \mathbb{Q}_+} c(r)\chi(r),\]
where $c(r) = \langle f(\chi),\chi(r)\rangle_{L^2(\mathbb{T}^\infty)}$. (The notation $\chi(r)$ is explained at the end of the introduction.)

\begin{proof}[Proof of Theorem~\ref{thmlin}] To see that the supremum cannot be smaller than $1/4$, it suffices to set $g(s)={p_j}^{-s}$ and $f(s)=1$. 

To prove the bound from above, we begin by expanding the function $h_p(\chi):=|f_\chi |^{p/2}$ in a Fourier series on $\mathbb{T}^{\infty}$,
\[ h_p(\chi)=\sum_{r\in \mathbb{Q}^+} c(r) \chi(r). \]
Using Lemma~\ref{lempoint} with $z_j=p_j^{-\sigma}\chi(p_j)$ and $\alpha=p/2$, we get that 
\[ |f_{\chi}(\sigma)|^{p/2} \le \int_{\mathbb{T}^d} h_p(w) P_{z}(w) dm_d(w)  = \sum_{(m,n)=1} c\left(\frac{m}{n}\right) (mn)^{-\sigma} \chi\left(\frac{m}{n}\right),\]
where we in the last step integrated the Fourier series of $h_p$ term by term against the Poisson kernel.
It follows that
\[I_\sigma :=\int_{\mathbb{T}^{\infty}}  | f_{\chi}(\sigma)|^p |g'_{\chi} (\sigma)|^2  dm_{\infty}(\chi) \le  \sum_{j,k=1}^d \sum_{\frac{m\mu}{n\nu}=\frac{p_j}{p_k}} \left|c\left(\frac{m}{n}\right) c\left(\frac{\mu}{\nu}\right)\right| (mn \mu \nu p_jp_k)^{-\sigma} |b_j b_k |\log{p_j} \log{p_k},\]
where it is understood that $(m,n)=1$ and $(\mu,\nu)=1$. By symmetry, we get $I_{\sigma} \le 2I_{\sigma,1} + 2I_{\sigma,2}$,
where 
\begin{align*} 
I_{\sigma,1} & :=\sum_{j,k=1}^d \sum_{\frac{m\mu}{n\nu}=\frac{p_j}{p_k}, \atop p_j|m,\, p_k|n } \left|c\left(\frac{m}{n}\right) c\left(\frac{\mu}{\nu}\right)\right| (mn \mu \nu p_jp_k)^{-\sigma} |b_j b_k |\log{p_j} \log{p_k}\\
I_{\sigma,2} & :=\sum_{j,k=1}^d \sum_{\frac{m\mu}{n\nu}=\frac{p_j}{p_k}, \atop p_j|m,\, p_k|\nu } \left|c\left(\frac{m}{n}\right) c\left(\frac{\mu}{\nu}\right)\right| (mn \mu \nu p_jp_k)^{-\sigma} |b_j b_k |\log{p_j} \log{p_k}.
\end{align*}
We estimate the contribution from these two sums separately. First, by the Cauchy--Schwarz inequality, we have 
\begin{align*}
	I_{\sigma,1} &\leq \Bigg(\sum_{j,k=1}^d \sum_{(m,n)=1, \atop p_j|m,\, p_k|n} \left|c\left(\frac{m}{n}\right)\right|^2 \frac{\log p_j \log p_k}{(mn)^{2\sigma}}\Bigg)^\frac{1}{2}\Bigg(\sum_{j,k=1}^d \sum_{(\mu,\nu)=1}\left|c\left(\frac{\mu}{\nu}\right)\right|^2 |b_j|^2 |b_k |^2  \frac{\log p_j \log p_k}{(p_jp_k)^{2\sigma}}\Bigg)^\frac{1}{2} \\
	&\leq\Bigg(\sum_{(m,n)=1} \left|c\left(\frac{m}{n}\right)\right|^2 \frac{\log m \log n}{(mn)^{2\sigma}}\Bigg)^\frac{1}{2}\Bigg(\sum_{(\mu,\nu)=1}\left|c\left(\frac{\mu}{\nu}\right)\right|^2 \sum_{j,k=1}^d |b_j|^2 |b_k |^2  \frac{\log p_j \log p_k}{(p_jp_k)^{2\sigma}}\Bigg)^\frac{1}{2},
\end{align*}
where we in the final inequality changed the order of summation in the first factor and used that $\sum_{p_j|m}\log{p_j} \leq \log{m}$. To compute the integrals, we will use the identity 
\[\int_{0}^{\infty} (\log a)^2 a^{-2\sigma} \sigma d\sigma=\frac{1}{4}, \] 
which is valid for every $a>0$. We use the Cauchy--Schwarz inequality again and take the two integrals into the respective sums, to deduce that
\[\int_0^\infty I_{\sigma,1}\,\sigma d\sigma \leq \left(\sum_{(m,n)=1}\left|c\left(\frac{m}{n}\right)\right|^2\,\frac{\log{m}\log{n}}{4(\log{mn})^2}\right)^\frac{1}{2}\left( \sum_{(\mu,\nu)=1}\left|c\left(\frac{\mu}{\nu}\right)\right|^2 \sum_{j,k=1}^d|b_j|^2 |b_k |^2\,\frac{\log{p_j}\log{p_k}}{4(\log{p_jp_k})^2}\right)^\frac{1}{2}.\]
The fractions with logarithms are bounded by $1/16$, so in total we get that
\[ \int_{0}^{\infty} I_{\sigma,1} \,\sigma d\sigma \le \frac{1}{16} \|g\|_{\mathcal{H}^2}^2 \|f\|_{\mathcal{H}^p}^p.\]
To estimate $I_{\sigma,2}$, we use the Cauchy--Schwarz inequality and change the order of summation:
\begin{align*}
	I_{\sigma,2} &\le \Bigg(\sum_{j,k=1}^d \sum_{(m,n)=1, \atop p_j|m} \left|c\left(\frac{m}{n}\right)\right|^2   
	|b_k|^2 \frac{(\log p_j)^2}{(mp_j)^{2\sigma}}\Bigg)^\frac{1}{2}\Bigg(\sum_{j,k=1}^d \sum_{(\mu,\nu)=1, \atop p_k|\nu}\left|c\left(\frac{\mu}{\nu}\right)\right|^2 |b_j|^2 \frac{(\log p_k)^2}{(\nu p_k)^{2\sigma}}\Bigg)^\frac{1}{2} \\
	&=\|g\|_{\mathcal{H}^2}^2 \Bigg(\sum_{(m,n)=1} \left|c\left(\frac{m}{n}\right)\right|^2 \sum_{p_j|m} \frac{(\log p_j)^2}{(mp_j)^{2\sigma}}\Bigg)^\frac{1}{2}\Bigg(\sum_{(\mu,\nu)=1} \left|c\left(\frac{\mu}{\nu}\right)\right|^2 \sum_{p_k|\nu } \frac{(\log p_k)^2}{(\nu p_k)^{2\sigma}}\Bigg)^\frac{1}{2}.
\end{align*}
The two factors are symmetrical, so by using the Cauchy--Schwarz inequality again we get
\begin{align*}
	\int_0^\infty I_{\sigma,2}\,\sigma d\sigma &\leq \|g\|_{\mathcal{H}^2}^2 \sum_{(m,n)=1} \left|c\left(\frac{m}{n}\right)\right|^2 \sum_{p_j|m} \frac{(\log{p_j})^2}{4(\log{mp_j})^2} \\
	&= \frac{\|g\|_{\mathcal{H}^2}^2}{4} \sum_{(m,n)=1} \left|c\left(\frac{m}{n}\right)\right|^2 \frac{1}{\log{m}}\sum_{p_j|m} \log{p_j} \left(2 + \frac{\log{m}}{\log{p_j}}+\frac{\log{p_j}}{\log{m}}\right)^{-1} \leq \frac{\|g\|_{\mathcal{H}^2}^2\|f\|_{\mathcal{H}^p}^p}{16},
\end{align*}
where we used that $\log{m}/\log{p_j}+\log{p_j}/\log{m}\geq2$ when $p_j|m$. Combining everything yields
\[\int_0^\infty I_\sigma \,\sigma d\sigma \leq 2\int_0^\infty I_{\sigma,1}\,\sigma d\sigma + 2\int_0^\infty I_{\sigma,2}\,\sigma d\sigma \leq \frac{1}{4} \|g\|_{\mathcal{H}^2}^2 \|f\|_{\mathcal{H}^p}^p. \qedhere\]
\end{proof}

% DUAL & HANKEL
\section{Comparison of $\mathcal{X}$ with other spaces of Dirichlet series of $\BMO$ type} \label{sec:hankel} 

\subsection{Hardy spaces $\mathcal{H}^p $ and $\BMOA(\mathbb{C}_0)$} Our initial motivation for studying $\mathbf{T}_g$ was to consider $\mathcal{X}=\mathcal{X}_2$ as a type of $\BMOA$-space for the range of Hardy spaces $\mathcal{H}^p$. From Theorem~\ref{thm:bmocond}, we have the following inclusions, which show that $\mathcal{X}$ is in every $\mathcal{H}^p$, for $0<p<\infty$.

\begin{cor} \label{cor:inclusions}
	We have the following inclusions,
	\[\mathcal{H}^\infty \subsetneq \BMOA(\mathbb{C}_0)\cap\mathcal{D} \subsetneq \mathcal{X} \subsetneq \bigcap_{0< p < \infty} \mathcal{H}^p.\]
\end{cor}
\begin{proof}
	The inclusions are all from Theorem~\ref{thm:bmocond}. That the first inclusion is strict follows from Theorem~\ref{thm:zetabmo}. The second inclusion was observed to be strict in the remark at the end of Section~\ref{sec:multip}, but it can also be deduced from any example in Section~\ref{sec:homogen}. The strictness of the last inclusion follows from Theorem~\ref{thm:homweights} and the fact that
	\begin{equation} \label{eq:pequiv}
	\|g\|_{\mathcal{H}^p} \asymp \|g\|_{\mathcal{H}^2}
	\end{equation}
	when $g$ is an $m$-homogeneous Dirichlet series, with
	implied constants depending on $m$ and $p$. To verify \eqref{eq:pequiv}, we argue as follows. Let $d(n)$ be the number of divisors of the positive integer $n$. By the extension of Helson's inequality discussed in \cite[Sec. 5]{BHS15} and \cite[Thm.~3]{Seip13}, there exist nonnegative number $\alpha$ and $\beta$, depending on $p$, such that 
	\begin{equation}\label{eq:uplow} \Bigg(\sum_{n=1}^\infty \frac{|a_n|^2}{[d(n)]^\alpha}\Bigg)^\frac{1}{2} \leq \Bigg\|\sum_{n=1}^\infty a_n n^{-s}\Bigg\|_{\mathcal{H}^p} \leq \Bigg(\sum_{n=1}^\infty |a_n|^2 [d(n)]^\beta\Bigg)^\frac{1}{2}.\end{equation}
	The key point is that if $\Omega(n)=m$, then $m+1\leq d(n) \leq 2^m$, proving \eqref{eq:pequiv}. (In fact, by a suitable application of H\"{o}lder's inequality, we can prove \eqref{eq:pequiv} using only the right inequality in \eqref{eq:uplow}.)
\end{proof}

In the next three subsections, we will compare $\mathcal{X}$ with two other analogues of $\BMOA$, namely the dual space $(\mathcal{H}^1)^\ast$ and the space $(\mathcal{H}^2\odot\mathcal{H}^2)^\ast$ of symbols generating bounded multiplicative Hankel forms. Let us first recall that neither of these spaces is contained in
\[\bigcap_{0< p < \infty} \mathcal{H}^p.\]
This follows immediately from a result of Marzo and Seip \cite{MS11}, which states that the Riesz projection $P$ on the polytorus is unbounded from $L^\infty(\mathbb{T}^\infty)$ to $H^4(\mathbb{D}^\infty)$. In fact, it is not even known whether $P(L^\infty(\mathbb{T}^\infty))$ is contained in $H^p(\mathbb{D}^\infty)$ for any $p>2$. Note that $P(L^\infty(\mathbb{T}^\infty))$ is naturally identified with $(\mathcal{H}^1)^\ast$, and that it is strictly continuously contained in $(\mathcal{H}^2\odot\mathcal{H}^2)^\ast$ \cite{OCS12}.

\subsection{Hankel forms}
Let us now consider the space of symbols $g$ such that the corresponding Hankel forms $\mathbf{H}_g$ are bounded. The form $\mathbf{H}_g$ is given by 
\begin{equation*}
	\mathbf{H}_g(fh) := \langle fh, g\rangle_{\mathcal{H}^2},
\end{equation*}
from which it is clear, by definition, that $\mathbf{H}_g$ is bounded if and only if $g$ is in $(\mathcal{H}^2\odot\mathcal{H}^2)^\ast$. Applying the product rule for derivatives, we find that 
\begin{equation}
	\label{eq:halfhankel} \mathbf{H}_g(fh) = f(+\infty)h(+\infty)\overline{g(+\infty)} + \langle 
	\partial^{-1}(f'h), g\rangle_{\mathcal{H}^2} + \langle 
	\partial^{-1}(fh'), g\rangle_{\mathcal{H}^2}, 
\end{equation}
where
\[\partial^{-1}f(s) := -\int_{s}^\infty f(w) \, dw.\]
The ``half-Hankel'' form
\begin{equation} \label{eq:halfhankelform}
	(f,h) \mapsto \langle
	\partial^{-1}(f'h), g\rangle_{\mathcal{H}^2}
\end{equation}
is bounded if and only if $g \in (
\partial^{-1}(
\partial\mathcal{H}^2\odot\mathcal{H}^2))^\ast$. It is clear from \eqref{eq:halfhankel} that 
\begin{equation} \label{eq:skewinc}
	(\partial^{-1}(
	\partial\mathcal{H}^2\odot\mathcal{H}^2))^\ast \subset (\mathcal{H}^2\odot\mathcal{H}^2)^\ast.
\end{equation}
Whether the inclusion in \eqref{eq:skewinc} is strict, is an open problem. It was observed in \cite{BP15} that it is equivalent to an interesting Schur multiplier problem.

\begin{cor} \label{cor:hankelinc}
	Suppose that the Volterra operator $\mathbf{T}_g$ acts boundedly on $\mathcal{H}^2$. Then the Hankel form $\mathbf{H}_g$ is bounded.
\end{cor}
\begin{proof}
	The Littlewood--Paley formula \eqref{eq:LPp} may be polarized, to obtain
	\begin{equation}
		\langle f, g \rangle_{\mathcal{H}^2} = f(+\infty)\overline{g(+\infty)} + \frac{4}{\pi} \int_{\mathbb{T}^\infty} \int_\mathbb{R} \int_0^\infty f_\chi'(\sigma+it) \overline{g_\chi'(\sigma+it)} \sigma\,d\sigma\,\frac{dt}{1+t^2}\,dm_\infty(\chi). \label{eq:LPpolar} 
	\end{equation}
	We find that
	\begin{equation*}
		\langle
		\partial^{-1}(f'h), g\rangle_{\mathcal{H}^2} = \frac{4}{\pi} \int_{\mathbb{T}^\infty} \int_\mathbb{R} \int_0^\infty f_\chi'(\sigma+it) h_\chi(\sigma + it) \overline{g_\chi'(\sigma+it)} \sigma\,d\sigma\,\frac{dt}{1+t^2}\,dm_\infty(\chi). 
	\end{equation*}
	Hence, it is clear from Theorem~\ref{thm:avcarl} that if $\mathbf{T}_g$ is bounded, then so is the form \eqref{eq:halfhankelform}. Thus we may complete the proof by using the inclusion \eqref{eq:skewinc}.
\end{proof}
On weighted Dirichlet spaces of the disc (including the Hardy space), even in the vector-valued setting, the boundedness of a half-Hankel form also implies the boundedness of the corresponding $T_g$ operator (see \cite{AP12}). However, by \cite[Lem.~10]{BP15}, a half-Hankel form on $\mathcal{H}^2$ generated by a symbol $g$ with positive coefficients is bounded if and only if $\mathbf{H}_g$ is bounded. Since the symbols of Theorem~\ref{thm:zeta} generate bounded Hankel forms for $\alpha \geq 1/2$, but not bounded $\mathbf{T}_g$ operators for $\alpha < 1$, this shows that the same relationship between the half-Hankel form and $\mathbf{T}_g$ does not hold in the present context. 

\subsection{The dual of $\mathcal{H}^1$} 
The most tractable sufficient condition for  $g(s) = \sum_{n\geq1} b_n n^{-s}$ to belong to $(\mathcal{H}^1)^\ast$ was put forward by Helson \cite{Helson06}: $g$ is in $(\mathcal{H}^1)^\ast$ if 
\begin{equation}
	\label{eq:helson} \sum_{n=1}^\infty |b_n|^2 d(n) < \infty, 
\end{equation}
where again $d(n)$ denotes the number of divisors of the integer $n$. In fact, Helson's result is stated in terms of the Hankel form $\mathbf{H}_g$ considered above. If $g$ satisfies \eqref{eq:helson}, then $\mathbf{H}_g$ is Hilbert--Schmidt. 
Note that, by a consideration of zero sets based on \cite[Thm. 2]{Seip13}, we can show that a  Dirichlet series $g$ satisfying \eqref{eq:helson} will not always be in $\BMOA(\mathbb{C}_{1/2})$.

The examples of $g$ in $\mathcal{X}_2$ considered in Sections~\ref{sec:multip} and \ref{sec:homogen} are easily seen to satisfy \eqref{eq:helson}. Moreover, we see that the symbols in Theorem~\ref{thm:zeta}, $1/2<\alpha<1$, are in $(\mathcal{H}^1)^\ast$, but not in $\mathcal{X}_2$. Hence $(\mathcal{H}^1)^\ast$ is not contained in $\mathcal{X}_2$, and it is tempting to conjecture that $\mathcal{X}_2 \subset (\mathcal{H}^1)^\ast$.

First, let us show how to construct a class of Dirichlet series in $(\mathcal{H}^1)^\ast \cap \mathcal{X}_2$ that do not satisfy \eqref{eq:helson}, showing that Helson's criterion is not well adapted to understanding Volterra operators.
\begin{thm}
	\label{thm:gcdpaley} Suppose that $\mathcal{N}=\{n_1,\,n_2,\,\ldots\,\}\subset\mathbb{N}\backslash\{1\}$ is a set with the property that $(n_j,n_k)=1$ if $j\neq k$. If 
	\begin{equation}
		\label{eq:gpaley} g(s) = \sum_{n \in \mathcal{N}} b_n n^{-s}, 
	\end{equation}
	then $\|\mathbf{T}_g\|_{\mathcal{L}(\mathcal{H}^2)}=\|g\|_{\mathcal{H}^2}$. Moreover, for $f(s)=\sum_{n\geq1} a_n n^{-s}$, we have
	\[\Bigg(|a_0|^2 + \sum_{n\in \mathcal{N}} |a_n|^2\Bigg)^\frac{1}{2} \leq \sqrt{2}\|f\|_{\mathcal{H}^1}.\]
\end{thm}

The second statement in the theorem yields $\|g\|_{(\mathcal{H}^1)^\ast}\leq\sqrt{2}\|g\|_{\mathcal{H}^2}$, by the Cauchy--Schwarz inequality applied to $\langle f, g \rangle_{\mathcal{H}^2}$. Define the integers $n_1: = 2$, $n_2: = 3\cdot5$, $n_3 := 7\cdot11\cdot13$, and so on. The set $\mathcal{N}:=\{n_1,\,n_2,\,\ldots\,\}$ satisfies the assumptions of Theorem~\ref{thm:gcdpaley}, but $d(n_j)=2^j$, so \eqref{eq:helson} is not always satisfied.
\begin{proof}
	[Proof of Theorem~\ref{thm:gcdpaley}] For the first statement, we simply observe that
	\[\sum_{n|N \atop n \in \mathcal{N}} \log{n} \leq \log{N},\]
	which allows us to follow the proof of Theorem~\ref{thm:linear} to obtain that every Dirichlet series of the form \eqref{eq:gpaley} satisfies $\|\mathbf{T}_g\| = \|g\|_{\mathcal{H}^2}$.
	
	For the second statement, fix some $n=n_j$, and set $d:=\omega(n)$, $m:=\Omega(n)$ and $\kappa: = \kappa(n)$. By Helson's iterative procedure \cite{Helson06}, it is sufficient to demonstrate that for $f$  in $H^1(\mathbb{D}^d)$,
	\begin{equation} \label{eq:helsonproto}
		\left(|a_0|^2 + \frac{1}{2}|a_\kappa|^2\right)^\frac{1}{2}\leq \|f\|_{H^1(\mathbb{D}^d)}.
	\end{equation}
	We begin with Carleman's inequality (see \cite{Vukotic03}),
	\[\Bigg(\sum_{k=0}^\infty \frac{|c_k|^2}{k+1}\Bigg)^\frac{1}{2}\leq\Bigg\|\sum_{k=0}^\infty c_k w^k\Bigg\|_{{H}^1(\mathbb{D})}.\]
	Setting $F(w) = \sum_{k\geq0}c_k w^k$, we use F.~Wiener's trick (see \cite{Bohr14}) with an $m$th root of unity, say $\varphi$, so that
	\[F_m(w^m) := \frac{1}{m}\left(F(w)+F(w\varphi)+F(w\varphi^2)+\cdots+F(w\varphi^{m-1})\right) = \sum_{k=0}^\infty c_{mk}w^{mk}.\]
	Clearly $\|F_m\|_{H^1(\mathbb{D})}\leq \|F\|_{H^1(\mathbb{D})}$, so we find from Carleman's inequality that
	\begin{equation} \label{eq:wienercarleman}
		\Bigg(\sum_{k=0}^\infty \frac{|c_{mk}|^2}{k+1}\Bigg)^\frac{1}{2} \leq \Bigg\|\sum_{k=0}^\infty c_k w^k\Bigg\|_{H^1(\mathbb{D})}.
	\end{equation} 
	Returning to our function $f$ in $H^1(\mathbb{D}^d)$, we let $f_k$ denote the $k$-homogeneous part of $f$ and decompose $f$ accordingly:
	\[f(z) = \sum_{k=0}^\infty f_k(z).\]	
	Substituting $z_j \mapsto w z_j$ for $1\leq j \leq d$, we find, using  Fubini's theorem, \eqref{eq:wienercarleman}, and Minkowski's inequality, that
	\[\left(\sum_{k=0}^\infty \frac{1}{k+1}\,\|f_{km}\|_{H^1(\mathbb{D}^d)}^2\right)^\frac{1}{2} \leq \int_{\mathbb{D}^d}\Bigg(\sum_{k=0}^\infty \frac{|f_{km}(z)|^2}{k+1}\Bigg)^\frac{1}{2}\, dm_d(z) \leq \|f\|_{H^1(\mathbb{D}^d)}.\]
	We retain only the two first terms in the sum on the left-hand side. The proof of \eqref{eq:helsonproto} is completed by noting that $\|f_0\|_{H^1(\mathbb{D}^d)}=|a_0|$ and that $|a_\kappa| \leq \|f_m\|_{H^1(\mathbb{D}^d)}$, where the latter inequality holds because $|\kappa|=\Omega(n)=m$. 
\end{proof} 

As for the question of whether $\mathcal{X}_2 \subset (\mathcal{H}^1)^\ast$, our best result is the following corollary of the characterization given in Theorem~\ref{thm:avcarl}. For its interpretation, one should recall that \eqref{eq:cpgn} implies that $\mathcal{X}_1\subset\mathcal{X}_2$. Hence, the corollary also motivates further interest in the question of whether $\mathcal{X}_2 = \mathcal{X}_p$ for all $p$, $0<p<\infty$.
\begin{cor}
	Suppose that the Volterra operator $\mathbf{T}_g$ acts boundedly on $\mathcal{H}^1$. Then $g$ is in $(\mathcal{H}^1)^\ast$. 
\end{cor}
\begin{proof}
	Let $f$ be a Dirichlet series in $\mathcal{H}^1$ and suppose that $f(+\infty)=0$. Let $g$ be  $\mathcal{X}_1$ and apply \eqref{eq:LPpolar} along with the Cauchy--Schwarz inequality, 
	\begin{align*}
		|\langle f, g \rangle_{\mathcal{H}^2}| &\asymp \left|\int_{\mathbb{T}^\infty} \int_\mathbb{R} \int_0^\infty f_\chi'(\sigma+it) \overline{g_\chi'(\sigma+it)} \sigma\,d\sigma\,\frac{dt}{1+t^2}\,dm_\infty(\chi)\right| \\
		&\leq \left(\int_{\mathbb{T}^\infty} \int_\mathbb{R} \int_0^\infty \frac{|f_\chi'(\sigma+it)|^2}{|f_\chi(\sigma+it)|} \sigma\,d\sigma\,\frac{dt}{1+t^2}\,dm_\infty(\chi)\right)^\frac{1}{2}\,\qquad\qquad\qquad \\
		&\qquad\qquad\qquad \times\,\left(\int_{\mathbb{T}^\infty} \int_\mathbb{R} \int_0^\infty |f_\chi(\sigma+it)|\,|g_\chi'(\sigma+it)|^2 \sigma\,d\sigma\,\frac{dt}{1+t^2}\,dm_\infty(\chi)\right)^\frac{1}{2}. 
	\end{align*}
	We finish the proof by using Theorem~\ref{thm:avcarl} with $p=1$, since the quantity on the second line is  bounded from above and below by $\|f\|_{\mathcal{H}^1}^{1/2}$ in view of the Littlewood--Paley formula \eqref{eq:LPp}. 
\end{proof}
Observe that by part (a) of Theorem~\ref{thm:bmoxp}, this shows in particular that if $g$ is in $\BMOA(\mathbb{C}_0)\cap\mathcal{D}$, then $g$ is in $(\mathcal{H}^1)^\ast$. This inclusion can also be deduced directly from the two Littlewood--Paley formulas \eqref{eq:LPp} and \eqref{eq:LPpolar}, using the Cauchy--Schwarz inequality and Lemma~\ref{lem:carlesonbmohalfplane}.

\subsection{On the finite polydisc $\mathbb{D}^d$} Let us now confine ourselves to studying Dirichlet series
\[f(s) = \sum_{n=1}^\infty a_n n^{-s}\]
restricted to the first $d$ primes, by demanding that $a_n=0$ if $p_j|n$, for $j>d$. Through the Bohr lift, the restricted Hardy spaces $\mathcal{H}^p_d$ (which are complemented subspaces of $\mathcal{H}^p$) are isometrically identified with $H^p(\mathbb{D}^d)$. We consider now a Dirichlet series $g$ restricted to the first $d$ primes and let $\mathbf{T}_g$ act on $\mathcal{H}^p_d$. 
\begin{cor}
	\label{cor:H^pd} For $0<p<\infty$, $\mathbf{T}_g$ is bounded on $\mathcal{H}^p_d$ if and only if it is bounded on $\mathcal{H}^2_d$. 
\end{cor}
\begin{proof}
	This follows from Theorem~\ref{thm:avcarl}, since the Carleson measure characterization is now over $\mathbb{D}^d$, and the Carleson measures of $H^p(\mathbb{D}^d)$ are independent of $p$ (see \cite{Chang79}). 
\end{proof}

Moreover, using the result that $H^2(\mathbb{D}^d)\odot H^2(\mathbb{D}^d)=H^1(\mathbb{D}^d)$ from \cite{FL02,LT09}, we conclude that symbols inducing bounded $\mathbf{T}_g$-operators on the finite polydisc belong to $(H^1(\mathbb{D}^d))^\ast$. This subsection is devoted to showing that, even in the finite-dimensional setting, the dual of $H^1$ still does not characterize the bounded $\mathbf{T}_g$-operators.

Let $D$ denote the differentiation operator on Dirichlet series,
\[Df(s) := f'(s) = -\sum_{n=2}^\infty a_n (\log n) n^{-s}.\]
Identifying again $\mathcal{H}^p_d$ with $H^p(\mathbb{D}^d)$, we find that we may write
\begin{equation}
	\label{eq:diridiff} Df(z_1,\,\ldots,\,z_d) = -\sum_{j=1}^d (\log p_j) z_j 
	\partial_{z_j} f(z_1,\,\ldots,\,z_d). 
\end{equation}
Note the similarity between $D$ and the radial differentiation operator 
\begin{equation}
	\label{eq:raddiff} Rf(z_1,\,\ldots,\,z_d) := \sum_{j=1}^d z_j 
	\partial_{z_j} f(z_1,\,\ldots,\,z_d). 
\end{equation}
The Volterra operator $T_g$ defined with the radial differentiation operator $R$ and radial integration $R^{-1}$ has previously been investigated on the unit ball $\mathbb{B}_d$ of $\mathbb{C}^d$ by a number of authors. A seminal contribution is that of Pau \cite{Pau13}, who proved that $T_g$ is bounded on $H^p(\mathbb{B}_d)$ if and only if $g$ is in $\BMOA(\mathbb{B}_d)$. In particular, for $p=2$, the $T_g$ operator is bounded if and only if the corresponding Hankel operator is bounded, i.e., if and only if $g$ defines a bounded linear functional on $H^2(\mathbb{B}_d) \odot H^2(\mathbb{B}_d)$. 

We shall now see that the corresponding statement is not true on the finite polydisc $\mathbb{D}^2$. The statement and proof are written for the Volterra operator defined in terms of radial differentiation \eqref{eq:raddiff}, but the argument works equally well for the half-plane differentiation \eqref{eq:diridiff}. In the following theorem, we use the notation
$g_1\otimes g_2(z,w):=g_1(z)g_2(w)$.
\begin{thm}
	\label{thm:finitecase} There exist a function $g_1$ in $H^\infty(\mathbb{D})$ and  a function $g_2$ in $\BMOA(\mathbb{D})$ such that $T_{g_1\otimes g_2}$ is unbounded on $H^2(\mathbb{D}^2)$. 
	\end{thm}
To obtain the desired conclusion from this theorem, namely that $T_g$ is not bounded simultaneously with the Hankel operator $H_g$ even on the bidisc, it suffices to observe that the symbol $g_1\otimes g_2$ is in $\BMOA(\mathbb{D}^2)$ and therefore in $\left( H^2(\mathbb{D}^2) \odot H^2(\mathbb{D}^2) \right)^* = \left(H^1(\mathbb{D}^2)\right)^*$. 

\begin{proof}[Proof of Theorem~\ref{thm:finitecase}]
	Suppose that $f(z,w) = \sum_{m,n\geq0} a_{m,n} z^m w^n$. Then
	\[Rf(z,w) = \sum_{m,n\geq0} (m+n) a_{m,n}z^m w^n \qquad \text{and} \qquad R^{-1} f(z,w) = \sum_{m,n\geq0\atop m+n>0} \frac{a_{m,n}}{m+n}z^m w^n.\]
	We consider the Volterra operator $T_gf = R^{-1}(fRg)$, choosing $f= f_1\otimes f_2$, where $f_1$ and $f_2$ are both in $H^2(\mathbb{D})$. We compute and find that
	\begin{equation}
		\label{eq:multdisc} f(z,w) Rg (z,w) = f_1(z)f_2(w)\left(zg'_1(z)g_2(w) + wg_1(z)g'_2(w)\right). 
	\end{equation}
	We consider first the second term of \eqref{eq:multdisc}, which we write as $h_1(z)h_2(w)$, where
	\[h_1(z) := f_1(z)g_1(z) = \sum_{m=0}^\infty a_m z^m \qquad \text{and} \qquad h_2(w) := wf_2(w)g'_2(w) = \sum_{n=1}^\infty b_n w^n.\]
	Since $f_1$ is in $H^2(\mathbb{D})$ and $g$ is in $H^\infty(\mathbb{D})$, clearly $h_1$ is in $H^2(\mathbb{D})$, so $\sum_{m\geq0} |a_m|^2 < \infty$. In a similar way, we see that $h_2$ is the derivative of a function in $H^2(\mathbb{D})$ because $f_2 $ is in $H^2(\mathbb{D})$ and $g_2$ is  in $\BMOA(\mathbb{D})$ so that the operator $T_{g_2}$ is bounded on $H^2(\mathbb{D})$. This means that $\sum_{n\geq1} |b_n|^2/n^2 < \infty$. We conclude therefore that
	\[\big\|R^{-1}(h_1h_2)\big\|_{H^2(\mathbb{D}^2)}^2 = \sum_{m=0}^\infty \sum_{n=1}^\infty \frac{|a_m|^2 |b_n|^2}{(m+n)^2} \leq \sum_{m=0}^\infty |a_m|^2 \sum_{n=1}^\infty \frac{|b_n|^2}{n^2}< \infty.\]
	
	Changing our attention to the first term in \eqref{eq:multdisc}, it remains for us to show that we can pick $f_1$, $f_2$, $g_1$, and $g_2$ satisfying our assumptions, so that the $H^2(\mathbb{D}^2)$-norm of
	\[R^{-1}\left(zf_1(z)g'_1(z)f_2(w)g_2(w)\right)\]
	is infinite. Replace for the moment $zf_1(z)g'_1(z)$ with an arbitrary function $h_1$ in $z
	\partial H^2(\mathbb{D})$, say
	\[h_1(z) = \sum_{m=1}^\infty a_m z^m.\]
	Choose $f_2$ and $g_2$ as
	\[f_2(w) = \sum_{n=2}^\infty \frac{w^n}{\sqrt{n}(\log{n})} \qquad \text{and} \qquad g_2(w) = -\log(1-w).\]
	The coefficients of $h_2(w): = f_2(w)g_2(w) = \sum_{n\geq3} b_n w^n$ are given by
	\[b_n = \sum_{k=2}^{n-1} \frac{1}{\sqrt{k}(\log{k})}\,\frac{1}{(n-k)} \gg \frac{1}{\sqrt{n}(\log{n})}\sum_{k=2}^{n-1} \frac{1}{n-k} \gg \frac{1}{\sqrt{n}}.\]
	Hence we find that
	\[\big\|R^{-1}(h_1h_2)\big\|_{H^2(\mathbb{D}^2)}^2 \gg \sum_{m=1}^\infty \sum_{n=3}^\infty \frac{|a_m|^2}{(m+n)^2n} \asymp \sum_{m=1}^\infty \frac{|a_m|^2\log(m+2)}{(m+1)^2} = \infty\]
	for an appropriate choice of $h_1$ in $z
	\partial H^2(\mathbb{D})$. However, by a factorization result of Aleksandrov and Peller \cite{AP96}, there exist $f_1^j$ in $H^2(\mathbb{D})$ and $g_1^j$ in $H^\infty(\mathbb{D})$ for $1 \leq j \leq 4$, such that
	\[h_1(z) = z\sum_{j=1}^4 f_1^j(z)(g_1^j)'(z). \]
	Therefore, at least one of the four pairs $(f_1^j, g_1^j)$, $1 \leq j \leq 4$, will do as the choice of $(f_1,g_1)$. 
\end{proof}

% COMPACT & RPK
\section{Compactness of $\mathbf{T}_g$ on $\mathcal{H}^2$} \label{sec:RPK}
\subsection{Basic results}
We turn to a brief discussion of compactness of  $\mathbf{T}_g$. Every polynomial symbol $g(s) = \sum_{n\leq N} b_n n^{-s}$ defines a compact $\mathbf{T}_g$-operator, since in this case $\mathbf{T}_g$ is the sum of $N$ diagonal operators with entries in $c_0$. This means that all bounded operators from Section \ref{sec:homogen} actually are compact. To see this, let $S_N$ denote the partial sum operator, acting on a Dirichlet series $f(s) = \sum_{n\geq1} a_n n^{-s}$ by
\[S_Nf(s) = \sum_{n=1}^N a_n n^{-s}.\]
Suppose now that we have an estimate of the type $\|\mathbf{T}_g\|^2 \leq \sum_{n\geq2} |b_n|^2 w(n)$ for some positive weight function $w(n)$. If the right hand side is finite for some Dirichlet series $g$, then
\[\|\mathbf{T}_g - \mathbf{T}_{S_Ng}\|^2 \leq \sum_{n \geq N} |b_n|^2 w(n) \to 0, \qquad N \to \infty, \]
demonstrating that $\mathbf{T}_g$ is compact. In particular, every bounded $\mathbf{T}_g$-operator with a linear symbol is compact, since then $\|\mathbf{T}_g\|_{\mathcal{L}(\mathcal{H}^2)}=\|g\|_{\mathcal{H}^2}$, by Theorem~\ref{thm:linear}. Let us also mention that the Volterra operator defined by the primitive of the zeta function considered in Theorem~\ref{thm:zetabmo},
\[g(s) = \sum_{n=2}^\infty \frac{1}{n\log{n}}n^{-s},\]
is compact by this argument and Theorem~\ref{thm:genweight}. In the next subsection, we will produce a concrete example of a non-compact operator, by testing the Volterra operator of Theorem~\ref{thmsuf}, for $\lambda=1$, against reproducing kernels for suitable subspaces of $\mathcal{H}^2$.
 
We mention that it is possible to prove versions of Theorems \ref{thm:bmocond}, \ref{thm:avcarl}, and \ref{thm:bmoxp} for compactness, by replacing bounded mean oscillation by vanishing mean oscillation, and embeddings by vanishing embeddings. The details are standard, see for instance \cite{Pau13} for the arguments in a different setting. 

We present only two results in this section. The first is that the closure of Dirichlet polynomials in $\BMOA(\mathbb{C}_0)$ is $\VMOA(\mathbb{C}_0) \cap \mathcal{D}$, as it relies on the translation invariance (i) of Lemma~\ref{lem:chibmo} enjoyed by Dirichlet series in $\BMOA(\mathbb{C}_0)$. Recall that $\VMOA(\mathbb{C}_0)$ consists of those $g\in \BMOA(\mathbb{C}_0)$ such that
\[\lim_{\delta \to 0^+}\sup_{|I| < \delta} \frac{1}{|I|}\int_{I}\left|f(it)-\frac{1}{|I|}\int_I f(i\tau)\,d\tau\right|\,dt = 0.\]
We endow the space $\BMO(\mathbb{C}_\theta)\cap\mathcal{D}$ with the norm $\|f\|_{\BMO(\mathbb{C}_\theta)\cap\mathcal{D}} := |f(+\infty)| + \|f\|_{\BMO(\mathbb{C}_\theta)}$.
\begin{thm}
	\label{thm:vmoacomp} Let $g$ be a symbol in $\VMOA(\mathbb{C}_0)\cap\mathcal{D}$ and $\varepsilon$ be a positive number. Then there is a Dirichlet polynomial $P$ such that $\|g-P\|_{\BMO(\mathbb{C}_\theta)\cap\mathcal{D}}<\varepsilon$. 
\end{thm}
\begin{proof}
	Let $B_\delta$ denote the horizontal shift operator given by $B_\delta g(s) = g(s+\delta)$, and, as above, let $S_N$ denote the partial sum operator. We choose $P = B_\delta S_N g$, for some $\delta>0$ and $N$ to be specified later. Clearly $P(+\infty)=b_1=g(+\infty)$. Since $g$ is in $\VMOA(\mathbb{C}_0)$, we know from \cite[Thm.~VI.5.1]{Garnett} that
	\[\lim_{\delta\to0}\|g - B_\delta g\|_{\BMO(\mathbb{C}_0)}=0.\]
	Choose $\delta>0$ so that $\|g - B_\delta g\|_{\BMO(\mathbb{C}_\theta)}<\varepsilon/2$. Then
	\[\|g - P\|_{\BMO(\mathbb{C}_0)} \leq \|g - B_\delta g \|_{\BMO(\mathbb{C}_0)} + \|B_\delta g - P\|_{\BMO(\mathbb{C}_0)} < \varepsilon/2 + 2\|B_\delta g - B_\delta S_N g\|_{\mathcal{H}^\infty}.\]
	Now, by (i) of Lemma~\ref{lem:chibmo}, we know that $\sigma_b(g)\leq 0$. By a theorem of Bohr \cite{Bohr13}, this implies that $S_N g(s)$ converges uniformly to $g(s)$ in the closed half-plane $\mathbb{C}_\delta$, for every $\delta>0$. Hence there is some $N=N(g,\delta)$ such that $\|B_\delta g - B_\delta S_N g\|_{\mathcal{H}^\infty}=\|B_\delta(g - S_N g)\|_{\mathcal{H}^\infty}< \varepsilon/4.$ 
\end{proof}

Our second basic result is that $\mathbf{T}_g$ is never in any Schatten class, unless $g$ is constant. This is in line with \cite[Thm.~6.7]{Pau13}, showing that a radial Volterra operator $T_g\neq0$ defined on $H^2(\mathbb{B}_d)$ can be in the Schatten class $S_p$ only for $p>d$.
\begin{thm}
	\label{thm:schatten} Let
	\[g(s) = \sum_{n=1}^\infty b_n n^{-s}\]
	be a non-constant Dirichlet series. Then $\mathbf{T}_g\colon\mathcal{H}^2\to\mathcal{H}^2$ is not in $S_p$, for any $p < \infty$. 
	\begin{proof}
		Since $g$ is not constant, we know there is at least one non-zero term, so set
		\[N = \inf\left\{n \geq 2 \, :\, b_n \neq 0\right\} < \infty.\]
		We will use \cite[Thm.~1.33]{Zhu} in the following way: Set $e_n(s): = n^{-s}$ and
assume that $2\leq p < \infty$. Then the set $\{e_n\}_{n\geq1}$ forms an orthonormal basis for $\mathcal{H}^2$, so that:
		\[\|\mathbf{T}_g\|_{S_p}^p \geq \sum_{n=N}^\infty \|\mathbf{T}_g e_n\|_{\mathcal{H}^2}^p.\]
		A simple computation shows that if $n \geq N$, then we have
		\[\|\mathbf{T}_g e_n\|_{\mathcal{H}^2}^2 = \sum_{m=2}^\infty \frac{|b_m|^2(\log{m})^2}{(\log{mn})^2} \geq \frac{|b_N|^2(\log{N})^2}{(\log{nN})^2} \geq \frac{|b_N|^2(\log{N})^2}{(2\log{n})^2}.\]
		In particular, $\|\mathbf{T}_g e_n\|_{\mathcal{H}^2} \geq (|b_N|\log{N})/(2\log{n})$ and hence $\|\mathbf{T}_g\|_{S_p}^p \geq \infty$. The inclusion between Schatten classes allows us to conclude that $\mathbf{T}_g$ cannot be in $S_p$ for any $0 < p < \infty$. 
	\end{proof}
\end{thm}

\subsection{Estimating $y$-smooth reproducing kernels} 
We will now study the action of $\mathbf{T}_g$ on reproducing kernels for suitable subspaces of $\mathcal{H}^2$. The reproducing kernel $k_w$ of $\mathcal{H}^2$ itself at $w$, where $\mre(w) > 1/2$, is given by
\[k_{w}(s) := \zeta\left(s+\overline{w}\,\right) = \prod_p \left(1-p^{-s-\overline{w}}\,\right)^{-1}. \]
Considering these reproducing kernels is insufficient in our analysis of the multiplicative symbol $g$ from Theorem \ref{thmsuf}. Indeed, regardless of the value of $\lambda$, the Dirichlet series $g(s)$ converges absolutely all the way down to $\mre(s) = \sigma > 0$. Testing $\mathbf{T}_g$ on the kernels $k_w$, in $\mathbb{C}_{1/2}$ is therefore not enough to detect that it is unbounded for $\lambda > 1$. 

To address this, we consider $y$-smooth reproducing kernels. Let $P^+(n)$ denote the largest prime factor of $n$. The integer $n$ is called $y$-smooth if $P^+(n)\leq y$. The $y$-smooth reproducing kernels, $k_w^y$ are defined for $\mre(w)>0$ and $y \geq 1$, by cutting off prime numbers larger than $y$. This means that we set $k_w^y(s) := \zeta\left(s+\overline{w}, y\right)$, where 
\[ \zeta\left(s+\overline{w}, y\right) := \prod_{p \leq y} \left(1 - p^{-s-\overline{w}}\right)^{-1}.\]

Notice that we already used another variant of cut-off kernels in the proof of Theorem \ref{thmsuf}. Following G\'al's construction, we tested against a finite-dimensional kernel at $\sigma=0$, cut off to be smooth (in the sense of primes) and retaining only suitable small powers of each prime. Our motivation for turning to the more involved investigation of the reproducing kernels $k_w^y(s)$ is that they provide slightly better estimates than the rougher argument stemming from G\'al's work. More specifically, we will see that the multiplicative symbol $g$ from Theorem \ref{thmsuf} with $\lambda=1$ provides the only concrete example of a non-compact $\mathbf{T}_g$-operator in this paper. As in Section~\ref{sec:multip}, we consider without loss of generality the operator $\widetilde{\mathbf{T}}_g$ instead of $\mathbf{T}_g$, the difference between the two being compact.

Suppose that $f(s) = \sum_{n\geq1} \varphi(n) n^{-s}$, where $\varphi$ is a non-negative completely multiplicative function and that $g(s) = \sum_{n\geq1} b_n n^{-s}$ has non-negative coefficients. A computation shows that 
\begin{equation}
	\label{eq:multcomp} \big\|\widetilde{\mathbf{T}}_g f\big\|_{\mathcal{H}^2}^2 = \sum_{m=2}^\infty\sum_{n=2}^\infty( b_m\log{m})(b_n\log{n})\varphi\left(\frac{mn}{(m,n)^2}\right)\sum_{k=1}^\infty \frac{\varphi(k)^2}{\left(\log k + \log\frac{mn}{(m,n)} \right)^2}. 
\end{equation}
We will now choose $f$ to be a $y$-smooth reproducing kernel and estimate the innermost sum.
\begin{lem}
	\label{lem:smoothest} Let $\varphi(n)$ be the completely multiplicative non-negative function defined by setting
	\[ \varphi(n):=\begin{cases} n^{-\sigma}, & \text{if } P^+(n)\le y, \\
	0, & \text{otherwise}. \end{cases} \]
	Fix $\alpha$, $0<\alpha<1$. If $y^{\alpha} \geq 1/\sigma$, then for sufficiently large $y$ (depending on $\alpha$), we have
	\[S_\varphi(m,n) := \sum_{k=1}^\infty \frac{\varphi(k)^2}{\left(\log k + \log\frac{mn}{(m,n)} \right)^2} \asymp \frac{\big\|k_\sigma^y\big\|_{\mathcal{H}^2}^2}{\left( \left(1+o(1)\right) (1-2\sigma)^{-1} y^{1-2\sigma} + \log\frac{mn}{(m,n)} \right)^2}, \]
	where $o(1)$ tends to $0$ as $y\to\infty$. 
\end{lem}
\begin{proof}
	We may assume that $0<\sigma<1/2$. Observe first that $\|k_\sigma^y\|_{\mathcal{H}^2}^2 = \zeta(2\sigma,y)$. For simplicity of notation, we write $a := \log\frac{mn}{(m,n)}$. By Abel summation, we see that  
	\[ S_\varphi(m,n) \sim 2\sigma \int_1^\infty \frac{\Psi(x,y)\,x^{-2\sigma}}{\left(\log x+a\right)^2}\,\frac{dx}{x}, \]
	where as usual $\Psi(x,y)$ denotes the number of $y$-smooth integers less than or equal to $x$. Observe that $\zeta(s,y)$ is the Mellin transform of $\Psi(x,y)$,
	\[\zeta(s,y) = s\int_0^\infty x^{s-1} \Psi(x,y) \, dx.\]
	Hence by writing $\Psi(x,y)$ as the inverse Mellin transform of $\zeta(s,y)$, integrating over the vertical line $\mre s = \xi$ for some $0<\xi<2\sigma$, and then changing the order of integration, we obtain 
	\begin{align*}
		I := \int_1^\infty \frac{\Psi(x,y)\,x^{-2\sigma}}{\left(\log x+a\right)^2}\,\frac{dx}{x} &= \int_1^\infty \left(\frac{1}{2\pi i} \int_{\xi-i\infty}^{\xi+i\infty} \zeta(s,y) x^s \,\frac{ds}{s}\right) \frac{x^{-2\sigma}}{\left(\log x+a\right)^2}\,\frac{dx}{x} \\
		&= \frac{1}{2\pi i} \int_{\xi-i\infty}^{\xi+i\infty} \zeta(s,y) \underbrace{\left(\int_1^\infty \frac{x^{s-2\sigma}}{\left(\log x+a\right)^2}\,\frac{dx}{x}\right)}_{J}\frac{ds}{s}. 
	\end{align*}
	By substituting $x = e^t$, using the identity
	\[ \frac{1}{(t+a)^2}=-\frac{d}{da}\int_{0}^\infty e^{-(t+a)x} dx, \] and interpreting the resulting integral as a Laplace transform, we find that
	\[J = \int_0^\infty \frac{e^{-t(2\sigma-s)}}{(t+a)^2}\,dt = -\frac{d}{da} \left(e^{2\sigma a}\int_{2\sigma}^\infty e^{-at}\mathcal{L}\left\{e^{s\cdot}\right\}(t)dt\right) = \int_{2\sigma}^\infty e^{-a(t-2\sigma)}(t-2\sigma)\frac{dt}{s-t}.\]
	Therefore, by changing the order of integration again, we obtain that
	\[I = \int_{2\sigma}^\infty e^{-a(t-2\sigma)}(t-2\sigma) \left(\frac{1}{2\pi i}\int_{\xi-i\infty}^{\xi+i\infty} \zeta(s,y)\frac{ds}{s(s-t)}\right)dt.\]
	We evaluate the inner integral by residues, capturing the simple pole in $s=t$, to see that
	\[I = \int_{2\sigma}^\infty e^{-a(t-2\sigma)}(t-2\sigma)\frac{\zeta(t,y)}{t}\,dt = \int_0^\infty \frac{\zeta(t+2\sigma,y)}{t+2\sigma}\,te^{-at}\,dt.\]
	Hence, to prove the statement of the lemma, we need to estimate
	\[\frac{2\sigma}{\zeta(2\sigma,y)}I = \frac{2\sigma}{\zeta(2\sigma,y)} \int_0^\infty \frac{\zeta(t+2\sigma,y)}{t+2\sigma}\,te^{-at}\,dt\]
	from below. Observe that
	\[ \frac{\zeta(t+2\sigma,y)}{\zeta(2\sigma,y)} \geq \exp\Bigg(-C t\sum_{p\leq y}p^{-2\sigma}\log p\Bigg) \geq \exp\left(-C(1-2\sigma)^{-1} t y^{1-2\sigma}\right) \]
	when, say, $t\leq 2y^{-\alpha}$. Here $1 < C = 1 + o(1)$. Assuming that $\sigma\geq y^{-\alpha}$, we have that $2\sigma/(t+2\sigma) \geq 1/2$, and we therefore obtain
	\begin{align*} \frac{2\sigma}{\zeta(2\sigma,y)} \int_0^\infty \frac{\zeta(t+2\sigma,y)}{t+2\sigma}\,te^{-at}\,dt
	& \gg \int_0^{2y^{-\alpha}} t \exp\left(-\left(a+C(1-2\sigma)^{-1}y^{1-2\sigma}\right)t \right)\,dt \\
	\ge  \frac{1}{2 \left(a+C(1-2\sigma)^{-1} y^{1-2\sigma} \right)^2} \end{align*}
	for sufficiently large $y$.
	On the other hand, the same type of estimates carried out in reverse order shows that 
	\begin{align*}
		\frac{2\sigma}{\zeta(2\sigma,y)} \int_0^\infty \frac{\zeta(t+2\sigma,y)}{t+2\sigma}\,te^{-at}\,dt &\ll \int_0^{\infty} t \exp\left(-\left(a+C'(1-2\sigma)^{-1}y^{1-2\sigma}\right)t \right)\,dt\\
		&= \frac{1}{\left(a+C'(1-2\sigma)^{-1} y^{1-2\sigma} \right)^2}, 
	\end{align*}
	where $1 > C' = 1 + o(1)$. 
\end{proof}

Applying \eqref{eq:multcomp} and Lemma~\ref{lem:smoothest} to a symbol of multiplicative type \eqref{eq:multsymb}, we find that 
\begin{equation}
	\label{eq:gksmooth} \frac{\big\|\widetilde{\mathbf{T}}_g k_\sigma^y\big\|_{\mathcal{H}^2}^2}{\big\|k_\sigma^y\big\|_{\mathcal{H}^2}^2} \asymp \sum_{P^+(m)\leq y} \sum_{P^+(n)\leq y}\psi(mn)\,\frac{(m,n)^{2\sigma}}{(mn)^\sigma}\left( \left(1+o(1)\right) (1-2\sigma)^{-1} y^{1-2\sigma} + \log\frac{mn}{(m,n)} \right)^{-2}. 
\end{equation}
under the assumptions on $y$ and $\sigma$ from Lemma~\ref{lem:smoothest}.
\begin{thm}
	\label{thm:noncompact} For $0 < \lambda < \infty$, let $g$ be the Dirichlet series \eqref{eq:multsymb}, where $\psi(n)$ is the completely multiplicative function defined on the primes by $\psi(p): = \lambda p^{-1}(\log p)$.  Fix $\alpha$, $0<\alpha<1$. If $\sigma = y^{-\alpha}$, then 
	\begin{equation}
		\label{eq:cpest} \frac{\big\|\widetilde{\mathbf{T}}_g k_\sigma^y\big\|_{\mathcal{H}^2}^2}{\big\|k_\sigma^y\big\|_{\mathcal{H}^2}^2} \gg y^{2(\lambda-1)}. 
	\end{equation}
	In particular, $\mathbf{T}_g$ is not compact when $\lambda=1$. 
\end{thm}
\begin{proof}
	Let $\mu(n$) denote the M\"obius function, the only property of which we need is that $\mu(n)=0$ unless $n$ is square-free. Restricting the sums in \eqref{eq:gksmooth} to square-free numbers and using that $(m,n)^{2\sigma}\geq1$, we find that 
	\begin{equation}\label{eq:inter}
		\frac{\big\|\widetilde{\mathbf{T}}_g k_\sigma^y \big\|_{\mathcal{H}^2}^2}{\big\|k_\sigma^y\big\|_{\mathcal{H}^2}^2} \gg \sum_{P^+(m)\leq y \atop \mu(m)\neq0} \sum_{P^+(n)\leq y \atop \mu(n)\neq0}\frac{\psi(mn)}{(mn)^\sigma}\left( \left(1+o(1)\right) (1-2\sigma)^{-1} y^{1-2\sigma} + \log\frac{mn}{(m,n)} \right)^{-2}. \end{equation} Now using that $m$ and $n$ are $y$-smooth and square-free, so that both $\log m$ and $\log n$ are bounded by $\pi(y)\log y\leq (1+o(1))y$ by the prime number theorem, we obtain from \eqref{eq:inter} that 
	\begin{align*} \frac{\big\|\widetilde{\mathbf{T}}_g k_\sigma^y \big\|_{\mathcal{H}^2}^2}{\big\|k_\sigma^y\big\|_{\mathcal{H}^2}^2}&  \gg	 \frac{1}{y^2}\sum_{P^+(m)\leq y \atop \mu(m)\neq0} \sum_{P^+(n)\leq y \atop \mu(n)\neq0}\frac{\psi(mn)}{(mn)^\sigma} \\
	& = \frac{1}{y^2} \sum_{P^+(m)\leq y \atop \mu(m)\neq0} \frac{\psi(m)}{m^\sigma}\sum_{P^+(n)\leq y \atop \mu(n)\neq0}\frac{\psi(n)}{n^\sigma}= \Bigg(\frac{1}{y}\sum_{P^+(m)\leq y \atop \mu(m)\neq0} \frac{\psi(m)}{m^\sigma}\Bigg)^2. 
	\end{align*}
	We may now complete the proof of the estimate \eqref{eq:cpest} by the following computation:
	\[\sum_{P^+(m)\leq y \atop \mu(m)\neq0} \frac{\psi(m)}{m^\sigma} = \prod_{p\leq y }\Bigg(1+\frac{\psi(p)}{p^\sigma}\Bigg) \asymp \exp\Bigg(\sum_{p\leq y} \frac{\psi(p)}{p^\sigma}\Bigg) \geq \exp\left(\frac{\lambda}{y^\sigma} \sum_{p\leq y}\frac{\log p}{p} \right) \asymp \exp\left(\frac{\lambda}{y^{\sigma}} \log{y}\right).\]
	In the last step, we used Mertens's first theorem, which asserts that $\sum_{p\leq y}\frac{\log p}{p}-\log y$ is bounded in absolute value by 2. Now \eqref{eq:cpest} follows because $y^{-\sigma}\log y=\log y +o(1)$ when $y\to \infty$ by our choice of $\sigma$.
	
	Finally, let $\{\sigma_j\}_{j\geq1}$ and $\{y_j\}_{j\geq1}$ be sequences such that $\sigma_j \to 0$ and $y_j \to \infty$ as $j\to\infty$. Then for every Dirichlet polynomial $P$, we have that $\langle P,\, k^{y_j}_{\sigma_j}\rangle_{\mathcal{H}^2}$ converges as $j\to\infty$. On the other hand, we have that $\|k^{y_j}_{\sigma_j}\|_{\mathcal{H}^2} \to \infty$. Therefore $k^{y_j}_{\sigma_j}/\|k^{y_j}_{\sigma_j}\|_{\mathcal{H}^2}$ converges weakly to $0$ in $\mathcal{H}^2$. Hence, the estimate shows, for suitably chosen $\sigma_j$ and $y_j$, that $\mathbf{T}_g$ is not compact for $\lambda=1$. 
\end{proof}

% ACKNOW
\section*{Acknowledgements}
The authors are grateful to Alexandru Aleman and Fr\'{e}d\'{e}ric Bayart for helpful discussions and remarks. They would also like to express their gratitude to the anonymous referee for a very careful review of the the paper. 

% BIBLIO
\bibliographystyle{amsplain} 
\bibliography{volterra} 

\end{document}